\documentclass[11pt]{amsart}
\usepackage{mathrsfs}
\usepackage{enumerate}
\usepackage{amssymb,amsmath,amsfonts,amsthm}
\usepackage{graphicx,subfigure,epstopdf}
\usepackage[usenames]{color}
\usepackage{url}
\usepackage{algorithm,algorithmic}
\usepackage{dsfont,verbatim}
\usepackage[colorlinks,linktocpage,linkcolor=blue]{hyperref}
\usepackage[margin=1.2in]{geometry}
\usepackage{enumerate,enumitem}
\usepackage[normalem]{ulem}
\usepackage{bm}
\usepackage{booktabs}
\usepackage{caption2}
\allowdisplaybreaks

\definecolor{darkred}{rgb}{.7,0,0}

\definecolor{green}{rgb}{0,0.7,0}

\newtheorem{theorem}{Theorem}[section]
\newtheorem{lemma}{Lemma}[section]
\newtheorem{proposition}{Proposition}[section]
\newtheorem{corollary}{Corollary}[section]

\newtheorem{remark}{Remark}[section]

\def\cC{\mathcal{C}}
\def\cM{\mathcal{M}}
\def\cLg{\mathcal{L}_\cM}
\def\bR{\mathbb{R}}
\def\ng{\nabla_\cM}

\def\e{\overline{e}}

\def\cI{\mathcal{I}}
\def\cLgh{\mathcal{L}_{\cM_h}}
\def\cK{\mathcal{K}}
\def\ngh{\nabla_{\cM_h}}
\def\vnu{\vec{\nu}}

\def\cBh{\mathcal{B}_h}
\def\phj{\psi_{j,h}}
\def\lhj{\lambda_{j,h}}
\def\Fto{{}_2F_1}
\def\cM{\mathcal{M}}
\def\cP{\mathcal P} 

\def\mH{\mathbb{H}}
\def\mL{\mathbb{L}}
\def\mbM{\mathbf{M}}
\def\mbS{\mathbf{S}}
\def\mbI{\mathbf{I}}
\def\mbR{\mathbf{R}}
\def\nell{\bar{\ell}}
\def\tell{l}
\def\tgamma{{\tilde{\gamma}}}
\def\tlambda{{\hat{\lambda}}}
\renewcommand{\theequation}{\thesection.\arabic{equation}}
\newcommand{\vertiii}[1]
{{\left\vert\kern-0.25ex\left
\vert\kern-0.25ex\left\vert #1
    \right\vert\kern-0.25ex\right
\vert\kern-0.25ex\right\vert}}

\newtheorem{example}{Example}[section]
\newtheorem{assumption}{\textsc{Assumption}}

\newcommand{\INDSTATE}[1][1]{\STATE\hspace{#1\algorithmicindent}}

\begin{document}
	
\title[Pad\'e-parametric FEM  approximation]{Pad\'e-parametric FEM approximation for fractional powers of  elliptic operators on manifolds}
%\thanks{Start date: Jan 26, 2016}}
\author{Beiping Duan}
\address{Faculty of Computational Mathematics and Cybernetics, Shenzhen MSU-BIT University, Shenzhen  518172, P.R. China.}
\email {duanbeiping@hotmail.com }
%\thanks{B. Duan is supported by  China Postdoctoral Science Foundation(2020M682895), and in part by grants NSFC 11871092 and NSAF U1930402.}
%Shenzhen JL Computational Science and Applied Research Institute, Shenzhen, P.R. China
%\date{\today}
\date{}


\begin{abstract}
	This paper focuses on numerical approximation for fractional powers of elliptic operators on $2$-d manifolds. Firstly, parametric finite element method is employed to discretize the original problem. We then approximate fractional powers of the discrete elliptic operator by the product of rational functions, each of which is a diagonal Pad\'e approximant for corresponding power function. Rigorous error analysis is carried out and sharp error bounds are presented which show that the scheme is robust for $\alpha\rightarrow 0^+$ and $\alpha \rightarrow 1^-$. The cost of the proposed algorithm is solving some elliptic problems. Since the approach is exponentially convergent with respect to the number of solves, it is very efficient. Some numerical tests are given to confirm our theoretical analysis and the robustness of the algorithm.
\end{abstract}
\maketitle

\section{Introduction}\label{Se:1}

Suppose $\cM$ is a $2$-dimensional compact and orientable manifold with $C^3$-smoothness in $\bR^{3}$ (see \cite[Section 2.1]{dziuk2013finite} for the definition of the smoothness), and let $\Gamma$ denote its boundary. We introduce the  self-adjoint operator $\cLg$ defined on $\cM$ by 
\begin{equation*}
	\int_{\cM}\cLg w\,v d\cM=\int_{\cM}a(\vec{x}) \ng w(\ng v)^T+b(\vec{x})w v\,d\cM,\quad \forall w,v\in \mH^1(\cM)
\end{equation*}
 where $\ng$ is the  gradient operator on $\cM$, which maps a scalar to a row vector. The space $\mH^1(\cM)$ shall be clarified later. For functions $a(\vec{x}), b(\vec{x}): \cM\rightarrow \mathbb R^+$, we assume along this paper that $a(\vec{x})\in H^1(\cM)$ with lower bound $\underline{a}>0$ and upper bound $\overline{a}$; $b(\vec{x})\in L^\infty(\cM)$ holds a.e. $ 0\le b(\vec{x})\le \bar{b}$.  For the Sobolev spaces and the  calculus on manifolds, one may refer to \cite{dziuk2013finite} or \cite{aubin1998some} for more details.
% \begin{itemize}
% 	\item $a\in C^2(\cM)$ and for any $\vec{x}\in \cM$,  $a(\vec{x})\in [\underline{a}, \overline{a}]$  with $\underline{a}>0$;
% 
%  	\item $b\in C^2(\cM)$ satisfies $b(\vec{x}) \in[0,\bar{b}]$.
% 	\item $b\in C^2(\cM)$ satisfies either $b(\vec{x}) \in[\underline{b},\bar{b}]$ with $\underline{b}>0$, or $b(\vec{x})=0$. 
% \end{itemize}

Consider the equation $\cLg w=f$, with $f$ square integrable on $\cM$. For $b(\vec{x})=0$ and $\Gamma=\emptyset$ it requires the compatibility condition $\int_{\cM}f d\cM=0$ since $$\int_{\cM}f d\cM=\int_{\cM}a(\vec{x}) \ng w(\ng 1)^T d\cM=0,$$
and we further require $\int_{\cM} w d\cM=0$ for the uniqueness.  So to unify the notations, we introduce the following symbols of normed spaces for later analysis:
  \begin{equation*}
 	\begin{aligned}
 		\mL^2(\cM)=\left\{
 		\begin{matrix}
 			L^2(\cM),  &  \mbox{if } \Gamma\neq \emptyset, \mbox{ or }\Gamma=\emptyset, b(\vec{x})\not\equiv 0 \\
 			\{v|v\in L^2(\cM), \int_{\cM}v \, d\cM=0\},  &  \mbox{if } \Gamma= \emptyset 	\mbox{ and }b(\vec{x})\equiv 0
 		\end{matrix}
 		\right.
 	\end{aligned}
 \end{equation*}
and
 \begin{equation*}
 	\begin{aligned}
 		\mH^1(\cM)=\left\{
 		\begin{matrix}
 			H_0^1(\cM),  &\quad \mbox{if } \Gamma\neq \emptyset \\
 			H^1(\cM), & \mbox{ if } \Gamma=\emptyset \mbox{ and } b(\vec{x})\not\equiv 0 \\
 			\{v|v\in H^1(\cM), \int_{\cM}v \, d\cM=0\},  &\mbox{ if } \Gamma= \emptyset 	\mbox{ and } b(\vec{x})\equiv 0
 		\end{matrix}
 		\right.
 	\end{aligned}
 \end{equation*}
and $\mH^2(\cM)=\mH^1(\cM)\cap H^2(\cM)$. Besides, we use $\mH^{-1}(\cM)$ to represent the dual space of $\mH^1(\cM).$

 Under the assumptions on $a(\vec{x})$ and $b(\vec{x})$ it follows that the bilinear form corresponding to $\cLg$ is continuous and coercive on $\mH^1(\cM)\times \mH^1(\cM)$. So thanks to Lax-Milgram lemma we know there exists a unique solution $w\in \mH^1(\cM)$.
 
  Suppose $\alpha\in(0,1)$. In this paper we investigate numerical approximation for the following problem:
\begin{equation}\label{eq-original}
\begin{aligned}
		\cLg^{\alpha} u&=f\quad \mbox{with } f\in \mL^2({\cM}),\\
%		u(x)|_{x\in \Gamma}&=0\quad \mbox{if }\Gamma\neq \emptyset,
\end{aligned}
\end{equation}
where $\cLg^{\alpha}$ represents the fractional power of $\cLg$ of  order $\alpha$. Note that $\cLg$ is a sectorial operator and in particular, is positive definite, so its negative powers are well defined on $\mL^2(\cM)$, given by Dunford-Taylor formula, see, e.g., \cite{Kato1961}, 
\begin{equation*}
	\cLg^{-\alpha}f = \frac {1}{2\pi i} \int _{\cC} {z^{-\alpha}}  (z\cI-\cLg)^{-1}    f \,dz
\end{equation*}
where $\cC$ is any contour on complex plane surrounds the spectra of $\cLg$ and excludes the negative
real axis.

\subsection{Motivation}
Nonlocal operators have a wide range of applications, especially in physics and image processing, see  e.g., \cite{constantin1999behavior,bakunin2008turbulence,eringen2002nonlocal,gilboa2008nonlocal,duvant2012inequalities,mikki2020theory}, so numerical approaches for problems involving such operators have attracted much attention. For more background or numerical approaches, one may refer to the recent review work \cite{delia2020numerical}.  As a special category of nonlocal operators, fractional powers of positive definite operators or more generally, of sectorial operators, are actually direct products of functional calculus. Numerical methods for such operators also receive more and more eyes, e.g., applying quadrature schemes on integral representations, utilizing the best uniform rational approximations and solving Caffarelli-Silvestre extension problem et al., and we recommend a comprehensive review paper \cite{harizanov2020survey} for interested readers. In fact all the schemes which seem quite diverse, ``can be interpreted as realizing different rational approximations of a univariate function over the spectrum of the original (nonfractional) diffusion operator'' \cite{hofreither2020a}.

 However, all the considerations in the work we mentioned above and the references therein is restricted in Euclidean spaces, and to our best knowledge, no work discusses numerical approximation for  fractional powers of elliptic operators on manifolds except \cite{bonito2021approximation}. In this very recent work, the authors generalized the Bonito-Pasciak quadrature scheme proposed in \cite{bonito2015numerical,bonito2016numerical}, which was obtained by applying trapezoidal rule for corresponding integral representation,  to Laplace-Beltrami operators on 2-d closed surfaces and presented full error analysis.  
 
 We would like to mention that, except pure  mathematical interest, we focus our attention on this issue partly because its application in 
 Gaussian random fields on manifolds, see e.g.,  \cite{bolin2020numerical,lang2015isotropic,lindgren2011an}. 

\subsection{Our approach}
The operator $ \cLg: \mH^1(\cM)\rightarrow \mH^{-1}(\cM)$ has real eigenfunctions and positive eigenvalues  $\{\lambda_{j},\psi_{j}\}_{j=1}^{\infty}$ due to its self-adjoint property and Poincar\'e inequality. And by the  compactness of $\cLg^{-1}$ from $\mH^{-1}(\cM)$ to $\mH^{1}(\cM)$, we know $\{\psi_j\}_{j=1}^{\infty}$ form complete orthogonal basis in $\mH^{-1}(\cM)$. Suppose $\{\psi_j\}_{j=1}^\infty$ are  normalized with respect to the $L^2(\cM)$-inner product. Then equivalently, for any $\alpha\in \bR^+$ 
we can define $\cLg^\alpha$ through the spectra of $\cLg$, namely, 
\begin{equation*}
	\cLg^\alpha  v:= \sum_{j=1}^{\infty} \lambda_j^\alpha (v,\psi_j)_{\cM} \,\psi_j \   \mbox{ for } \   v \in D(\cLg^\alpha),
\end{equation*}
where $(\cdot,\cdot)_{\cM}$ denotes the $L^2$-inner product on $\cM$, $D(\cLg^\alpha)$ is the domain of $\cLg^{\alpha}$ given by
\begin{equation*}
	D(\cLg^\alpha):= \left\{ v \in \mL^2(\cM)\Big| \ \sum_{j=1}^{\infty} \lambda_j^{2\alpha} |(v,\psi_j)_\cM|^2<\infty \right\}.
\end{equation*}
We use $\dot \mH^\gamma(\cM)$ to denote $D(\cLg^{\gamma/2})$ equipped with the norm
$$\|v\|_{\dot \mH^{\gamma}(\cM)} = \left (\sum_{j=1}^\infty \lambda_j^{\gamma} |(v,\psi_j)_\cM|^2 \right )^{1/2}.$$
For $\gamma\in [0,2]$ we denote by $\mH^\gamma(\cM)$ the intermediate space between $\mL^2(\cM)$ and $\mH^2(\cM)$ obtained from the real interpolation method, then $\mH^\gamma(\cM)$  and $\dot \mH^\gamma(\cM)$ coincide with equivalent norms(see \cite{guermond2009the,bonito2015numerical}). In fact, by considering the auxiliary equation $\cLg w=f$, one can easily verify that $\cLg^{-1}$ is a bounded map of $\mL^2(\cM)$ into $\mH^2(\cM)$ under our previous assumptions on $a(\vec{x})$ and $b(\vec{x})$. Further, it is also easy to check $\cLg$ is bounded from $\mH^2(\cM)$ to $\mL^2(\cM)$. So thanks to \cite[Proposition 4.1]{bonito2015numerical} the desired equivalence follows.

Assume that $\cM$ is approximated by union of finitely many non-degenerate $2$-dimensional simplices, say $\cM_h=\cup_{K\in \cK}\{K\}$ with $\cK$ the collocation of all the simplices. Then we use $\mL^2(\cM_h)$ and $\mH^1(\cM_h)$ to denote the discrete counterparts of $\mL^2(\cM)$ and $\mH^1(\cM)$, respectively. We let $\cP(\cM_h)$ denote the space consisting of piecewise linear polynomials on $\cM_h$ and $V_h:=\cP(\cM_h)\cap \mH^1(\cM_h)$.  Denote by $h_K$ the diameter of $K\in\cK$ and suppose that $h=\max_{K\in\mathcal{K}}h_K$ is small enough, then for any $\vec{x}\in \cM_h$ there exists a unique point $\theta(\vec{x})\in \cM$ such that 
\begin{equation}\label{eq-lift}
	\vec{x}=\theta(\vec{x})+d(\vec{x})\vnu(\theta(\vec{x}))
\end{equation}
with $\vnu$ the  unit outward normal vector of $\cM$ and $d(x)$ the oriented distance function. We refer $\theta(\vec{x})$ the lift of $\vec{x}$ onto $\cM$. The $C^3$-smoothness of $\cM$ implies that $\theta(\vec{x})\in C^2(\cM)$, which shall be used in Section \ref{Sec-4}.  Moreover, for any function $\xi$  defined on $\cM_h$, we denote $\xi^\ell=\xi\circ \theta^{-1}$, and for any $\bar{\xi}$ defined on $\cM$ we denote $\bar{\xi}^{-\ell}=\bar{\xi}\circ \theta$. 

Our approach can be split into two steps. We first approximate $\cLg^{-\alpha}$ by $\cLgh^{-\alpha}$ with $\cLgh$ the parametric finite element approximation of $\cLg$ satisfying for any $w_h, v_h\in V_h$
\begin{equation}\label{paraFEM-1}
	\int_{\cM_h}\cLgh w_h\,v_h d\cM_h=\int_{\cM_h}a^{-\ell}(\vec{x})\ngh w_h(\ngh v_h)^T+b^{-\ell}(\vec{x}) w_h v_h d\cM_h,
\end{equation}
where we implicitly assumed that the integral can be evaluated with desired accuracy. In fact, if $a(\vec{x})$ and $ b(\vec{x})$ possess more regularity one can adopt the following scheme:
\begin{equation}\label{paraFEM-2}
	\int_{\cM_h}\cLgh w_h\,v_h d\cM_h=\int_{\cM_h}a_h(\vec{x})\ngh w_h(\ngh v_h)^T+b_h(\vec{x}) w_h v_h d\cM_h,
\end{equation}
where $a_h=I_h (a^{-\ell})$, $b_h=I_h (b^{-\ell})$ with $I_h$ the piecewise  linear  interpolation operator on $\cM_h$. Compared with \eqref{paraFEM-1}, scheme \eqref{paraFEM-2} is more practical since it only needs the values of $a(\vec{x})$ and $b(\vec{x})$ at the vertexes of the mesh. Furthermore, the mass and stiffness matrices can be obtained accurately by simple calculus.  We shall prove that for given $f\in\dot{\mH}^\delta(\cM)$, if $a(\vec{x}),b(\vec{x})\in C^{\kappa}(\cM)$ with $\kappa=\min\{2\alpha+\delta,2\}$, then \eqref{paraFEM-2} can achieve the same convergence rate as \eqref{paraFEM-1}. Here $C^{\kappa}(\cM):=[C^0(\cM),C^2(\cM)]_{\kappa/2,\infty}$ is the real interpolation space between $C^0(\cM)$ and $C^2(\cM)$.

Different from  Euclidean cases, the approximation of $\cLg$ by $\cLgh$ involves both approximation error and geometric error. For $\cLgh$ we introduce its eigenvalues and eigenfunctions $\{\lhj,\phj\}_{j=1}^{D}$ and order them such that $\lambda_{j-1,h}\le \lhj$. Denote $\lambda_{h,min}=\min_j\{\lambda_{j,h}\}$ and  $\lambda_{h,max}=\max_j\{\lambda_{j,h}\}$. To avoid evaluating $\cLgh^{-\alpha}$ directly by eigenvalue-eigenfunction decomposition we reformulate it into
\begin{equation}\label{prod-L}
\begin{aligned}
	\cLgh^{-\alpha}&=\frac{1}{\tlambda^\alpha}\prod_{k=0}^{L}\left[(\tlambda\cI+t_{k+1}\cBh)(\tlambda\cI+t_k\cBh)^{-1}\right]^{-\alpha}\\
	&=\frac{1}{\tlambda^\alpha}\prod_{k=0}^{L}\left[1+\tau_k\cBh(\tlambda\cI+t_k\cBh)^{-1}\right]^{-\alpha}
\end{aligned}
\end{equation}
where $0=t_0<t_1<\cdots<t_{L}<t_{L+1}=1$, $\tau_k=t_{k+1}-t_k$, $\tlambda\in (0,\lambda_{h,min}]$ and $\cBh=\cLgh-\tlambda\cI$. We then approximate each term appeared in the product of \eqref{prod-L} by rational functions, say, using the following approximation
$$\left[1+\tau_k\cBh(\tlambda\cI+t_k\cBh)^{-1}\right]^{-\alpha}\approx r_m(\tau_k\cBh(\tlambda\cI+t_k\cBh)^{-1}),$$ 
which reduces the evaluation of $\cLgh^{-\alpha}f_h$ to solving $m(L+1)$ elliptic problems, with $f_h\in V_h$ a proper approximation of $f$. The reason we use the formulation \eqref{prod-L} rather than approximate 	$\cLgh^{-\alpha}$ directly is that the function $z^{-\alpha}$ is singular at $z_0=0$ and compared with the length of the target interval $[\lambda_{h,min},\lambda_{h,max}]$, $z_0$ locates too close to the left of the interval which can lead to bad approximation. One can understand this by normalizing the interval $[\lambda_{h,min},\lambda_{h,max}]$ to $[0,1]$. Different from the approaches obtained by integral formulas which are all of $ (n-1,n)$-type rational approximants, our scheme is of type $(n,n)$. In fact, the formulation \eqref{prod-L} is inspired by \cite{Vabishchevich14,Vabishchevich2015JCP,DLP-Pade}. In \cite{DLP-Pade} we chose Pad\'e approximation of $(1+t)^{-\alpha}$ as $r_m(t)$ and proved the optimal convergence order. In fact, the scheme in \cite{DLP-Pade} is a scheme of $h$-$p$ version and what we proved in that paper is in $h$-direction. In this paper we shall present the error in $p$-direction  and instead of Euclidean space, we shall focus our attention on manifolds.

\subsection{Our contributions}

\begin{itemize}%[(i)]
	\item 	In Section \ref{Sec-2} we establish sharp error bound for diagonal Pad\'e approximation of the real scalar function $(1+t)^{-\alpha}$: 
	\begin{equation*}
		0<r_m(t)-(1+t)^{-\alpha}\le  {c}'_{\alpha}\frac{2^{-4m}t^{2m+1}}{2^{mt}} ,\quad t\in[0,1]
	\end{equation*}
	with $r_m$ the $(m,m)$-type Pad\'e approximant and  $c'_\alpha\approx\frac{\alpha\pi}{2\Gamma(1-\alpha)\Gamma(1+\alpha)}$. Here $c_1(\alpha)\approx c_2(\alpha)$ means $\frac{c_1(\alpha)}{c_2(\alpha)}=\mathcal{O}(1)$.
	\item Based on this result, we then prove in Section \ref{Sec-3} that under optimal choice of  $\{t_k\}_{k=0}^{L+1}$,
	\begin{equation*}
	\begin{aligned}
		&\Big\|\cLgh^{-\alpha} f_h -\frac{1}{\tlambda^\alpha}\prod_{k=0}^{L}r_m(\tau_k\cBh(\tlambda\cI+t_k\cBh)^{-1}) f_h\Big\|_{L^2(\cM_h)}\\
		 &\le \hat{c}\tlambda^{-\alpha}\, 32^{-\frac{N_s}{\lceil \log_2 (\lambda_{h,max}/\tlambda)\rceil}}\|f_h\|_{L^2({\cM_h})}
	\end{aligned}
	\end{equation*}
where $N_s$ is the number of total solves, $\hat{c}\,\approx \frac{(\alpha+2)2^{\alpha-1}\pi}{\Gamma(1-\alpha)\Gamma(1+\alpha)}$. Note $\hat{c}$ is bounded with respect to $\alpha$, which  implies our scheme is robust for $\alpha\in (0,1)$.
	\item In Section \ref{Sec-4} we give the error between $\cLg^{-\alpha} f$ and $\cLgh^{-\alpha}f_h$ which includes the approximation error and geometric error: for $f\in \dot{\mH}^{\delta}(\cM)$
	\begin{equation*}
	\|u-u_h^\ell\|_{L^2(\cM)}\le 
	\left\{
	\begin{matrix}
		c |2\alpha+\delta-2|^{-1}h^{\min(\delta+2\alpha,2)}\|f\|_{\dot{H}^{\delta}(\cM)}& \quad \alpha+\delta/2\neq 1,\\
		c \ln(h^{-1})h^{2}\|f\|_{\dot{H}^{\delta}(\cM)}& \quad \alpha+\delta/2= 1,
	\end{matrix}
	\right.
\end{equation*}
where $c$ is independent of $h$ and is bounded for $\alpha\rightarrow 0^+$ and $\alpha \rightarrow 1^-$. Our  estimate for $\alpha+\delta/2\neq 1$ is sharper than the bound given in \cite[Thereom 4.2]{bonito2021approximation}. In their estimate, $\ln(h^{-1})$ is involved. 
\end{itemize}
%As was pointed out in \cite{HLM2020}, 
Compared with  Bonito-Pasciak scheme, both the cost and error bound are uniform with respect to $\alpha\in (0,1)$. In fact, As a function of total solves $N_s$, the error of Bonito-Pasciak scheme is of $\mathcal{O}(e^{-\pi\sqrt{\alpha(1-\alpha)N_s}})$,  say, it is root-exponentially convergent with respect to $N_s$,  while ours is $\mathcal{O}(32^{-N_s/\ln\lambda_{h,max}})$. Besides, one may also note that Bonito-Pasciak scheme is $\alpha$-dependent, and when $\alpha$ is taken close to  $0$ or $1$, the error will decay much slowly.  We also would like to mention the Jacobi-Gauss quadrature scheme proposed in \cite{aceto2019rational}, which yields an  $\mathcal{O}(\rho^{-N_s/\lambda_{h,\max}^{1/4}})$ convergency. Note that $\lambda_{h,\max}=\mathcal{O}(h^{-2})$, so for small enough $h$, $\ln\lambda_{h,max} \ll \lambda_{h,max}^{1/4}$ which implies our scheme converges faster. %Numerical tests also verify our assertion.
%%%%%%%%%%%%%%%%%%%%%%%%%%%%%%%%%%%%%%%%%%
%%%%%%%%%%%%%%%%%%%%%%%%%%%%%%%%%%%%%%%%%

%to obtain desired accuracy $\mathcal{O}(e^{-\pi^2/(2\tau)})$ with $\tau$ the quadrature step size, Bonito-Pasciak scheme needs $N_s=\lceil\frac{\pi^2}{4\alpha \tau^2}\rceil+\lceil\frac{\pi^2}{4(1-\alpha) \tau^2}\rceil+1$ solves under optimal choice, which will blow up when $\alpha\rightarrow 0^{+}$ or $\alpha\rightarrow 1^-$. Besides, as a function of total solves,  the error of Bonito-Pasciak scheme decays with a rate of $\mathcal{O}(\rho^{-\sqrt{N_s}})$,

\section{Pad\'e approximation of $(1+t)^{-\alpha}$}\label{Sec-2}
Pad\'e approximation is an old topic in approximation theory. Given a function $Z(t)$, the goal of Pad\'e approximation is to find  a rational function of the form $r=P_m(t)/Q_n(t)$ as the approximant of  $Z(t)$ with $P_m,Q_n$ polynomials of order no more than $m$ and $n$ respectively, such that the Taylor expansions of $Z(t)$ and $r(t)$ have the same coefficients up to $t^{m+n}$. 

In this section, we develop diagonal Pad\'e
approximations to  $(1+t)^{-\alpha}$ for $\alpha\in (0,1)$ based on the
classical theory of Pad\'e approximations given by Baker \cite{Baker1975}.  Our approximants are of the form
\begin{equation}\label{pade}
	(1+t)^{-\alpha}\approx r_m(t):=\frac{P_m(t)}{Q_m(t)}, \quad t\in [0,1]
\end{equation}
with $m=1,2,\ldots$ and normalized condition $Q_m(0)=1$. $r_m$ is determined by the Maclaurin expansion of $(1+t)^{-\alpha}$. 

We shall write down explicit formulas for the polynomials $P_m(t)$ and $Q_m(t)$ associated with orthogonal polynomials. The starting point is
\begin{equation}\label{taylor-alpha}
	(1+t)^{-\alpha} =1+\sum_{j=1}^{\infty}\frac{(-\alpha)(-\alpha-1)\cdots (-\alpha-j+1)}{j!}t^j= \Fto (\alpha,1;1;-t).
\end{equation}
Here $\Fto (a,b;c;t)$ denotes the hypergeometric function defined by
$$\Fto (a,b;c;t)=\sum_{j=0}^\infty \frac {(a)_j (b)_j} {j! (c)_j} t^j$$
with  $(a)_0=1$, $(a)_j=a(a+1)\cdots(a+j-1)$ for $j>0$. Clearly, \eqref{taylor-alpha} implies 
\begin{equation}\label{taylor-alpha1}
	(1-t)^{-\alpha} = \Fto (\alpha,1;1;t):=\sum_{i=0}^\infty c_i t^i,
\end{equation}
where
\begin{equation}\label{ci}
	\begin{aligned}
		c_i &= \frac {(\alpha)_i}{i!}=\frac 1 {\Gamma(\alpha)\Gamma(1-\alpha)} \int_0^1 
		(1-t)^{-\alpha} t^{\alpha+i-1} dt\\
		&=\frac 1 {\Gamma(\alpha)\Gamma(1-\alpha)} \int_0^1 t^{i-1}
		w^{-\alpha,\alpha}(t)\, dt
	\end{aligned}
\end{equation}
with $w^{\beta,\gamma}(t):= (1-t)^\beta t^\gamma$.

\begin{lemma}\label{lem-PQ}
	Let $\{p_i\}$ and $\{q_i\}$ denote two families of monic polynomials which are mutually orthogonal with respect to the weight $w^{\alpha,-\alpha}(t)$ and $w^{-\alpha,\alpha}(t)$ respectively. Then 
	\begin{equation*}
		P_m(-t)=t^mp_m(1/t), \qquad Q_m(-t)=t^mq_m(1/t).
	\end{equation*} 
\end{lemma}
\begin{proof}
	Denote $w(t)=\frac{w^{-\alpha,\alpha}(t)}{\Gamma(\alpha)\Gamma(1-\alpha)}$. Utilizing the orthogonality one can get 
	\begin{equation*}
		q_m(t)=(c_m^\alpha)^{-1} \,{\det \left| 
			\begin{matrix}
				\int_{0}^{1}wdt  &  \int_{0}^{1}twdt& \cdots & \int_{0}^{1}t^mwdt\\
				\int_{0}^{1}twdt &  \int_{0}^{1}t^2wdt& \cdots & \int_{0}^{1}t^{m+1}wdt\\
				\vdots           &   \cdots       &  \cdots      &\vdots \\
				\int_{0}^{1}t^{m-1}wdt &  \int_{0}^{1}t^mwdt& \cdots & \int_{0}^{1}t^{2m-1}wdt\\
				1                      & t                  &\cdots  & t^m
			\end{matrix}
			\right|}
	\end{equation*}
	with 
	\begin{equation*}
		c_m^\alpha={\det\left|  
			\begin{matrix}
				\int_{0}^{1}wdt  &  \int_{0}^{1}twdt& \cdots & \int_{0}^{1}t^{m-1}wdt\\
				\int_{0}^{1}twdt &  \int_{0}^{1}t^2wdt& \cdots & \int_{0}^{1}t^{m}wdt\\
				\vdots           &   \cdots       &  \cdots      &\vdots \\
				\int_{0}^{1}t^{m-1}wdt &  \int_{0}^{1}t^mwdt& \cdots & \int_{0}^{1}t^{2m-2}wdt
			\end{matrix}			
			\right|}.
	\end{equation*}
	It is obvious that $\frac{P_m(-t)}{Q_m(-t)}$ is the Pad\'e approximation of $(1-t)^{-\alpha}$. 
	Thanks to C. Jacobi's result(see, e.g. \cite[(1.27)]{Baker1975}), we know the denominator satisfies
	\begin{equation*}
		Q_m(-t)=\tilde{c}_m^\alpha\det\left|
		\begin{matrix}
			c_1   & c_2       & \cdots & c_{m+1}\\
			\vdots&           &       & \vdots\\
			c_m   & c_{m+1}   &\cdots & c_{2m}\\
			t^m   & t^{m-1}   &\cdots  & 1
		\end{matrix}
		\right|
	\end{equation*}
	where $\tilde{c}_m^\alpha$  is the constant such that $Q(0)=1$, and $\{c_i\}_{i=1}^{2m}$ are the coefficients in \eqref{taylor-alpha1}. Comparing $q_m(t)$ with $Q_m(-t)$ and recalling that $q_m$ is monic, one can immediately obtain
	$$Q_m(-t)=t^mq_m(1/t).$$
%	one can observe that if we set $$w(t)=\frac{w^{-\alpha,\alpha}(t)}{\Gamma(\alpha)\Gamma(1-\alpha)}$$ 
%	then $Q_m(-t)=t^m\Psi(1/t)$, where we also utilized $\Psi$ is monic. Thus we obtain $$Q_m(-t)=t^mq_m(1/t).$$
	Theorem 9.2 of \cite{Baker1975} implies that $Q_m(-t)/P_m(-t)$ is the diagonal Pad\'e approximation to $(1-t)^\alpha$. So by considering $\{c_i\}_{i=1}^\infty$ as the Maclaurin expansion coefficients of $(1-t)^\alpha$ we get
	\begin{equation*}
		c_i=\frac{1}{\Gamma(-\alpha)\Gamma(1+\alpha)}\int_{0}^{1}t^{i-1}w^{\alpha,-\alpha}(t)dt.
	\end{equation*}
	Repeating the same arguments for $p_m(t)$ and $P_m(-t)$ gives $P_m(-t)=t^mp_m(1/t)$. Thus we end our proof.
\end{proof}

Let $J_m^{\beta,\gamma}(t)$ denote the Jacobi Polynomial on $[0,1]$ with $\beta,\gamma\in \bR$, given by
\begin{equation*}
	J_m^{\beta,\gamma}(t)=\frac{(\beta+1)_m}{m!}{}_2F_1(-m,m+\beta+\gamma+1;\beta+1;1-t).
\end{equation*}
Suppose $\{t_j(\beta,\gamma)\}_{j=1}^{m}$ are the roots of $J_m^{\beta,\gamma}(t)$ enumerated in increasing order, then Lemma \ref{lem-PQ} implies that
\begin{equation}\label{formulation-rm}
	r_m(t)=\prod_{i=1}^m
	\frac {1+t_i(
		\alpha,-\alpha)t} {1+t_i(
		-\alpha,\alpha)t}.
\end{equation}

%\begin{proposition} The rational function
%	$r_m(t)$ is monotonically decreasing for $t\in [0,\infty)$.
%\end{proposition}
%\begin{proof} 
%	It follows by
%	\eqref{mark} that each term in the above product is monotonically decreasing and so
%	must the product.
%\end{proof}

%Suppose $t_1<t_2<\cdots< t_m $ are the roots of $P_m(t)$ and $t'_1<t'_2<\cdots< t'_m $ are those  of $Q_m(t)$, then  we have
%\begin{equation} \label{eqn:interlace}
%	-\infty<t_1<t'_1<t_2<t'_2\cdots<t_m<t'_m<-1.
%\end{equation}
%Since the roots of $P_m(t)$ and $Q_{m}$ are the negative reciprocals of
%those of $J_m^{\alpha,-\alpha}(t)$ and $J_m^{-\alpha,\alpha}(t)$, respectively, the interlacing property above is equivalent to the following statement.
To ensure the stability of our scheme, the following interlacing property is needed.

\begin{proposition}\label{interlace}
For  $\alpha\in (0,1)$ it holds
\begin{equation*}
	\begin{aligned}
		0&<t_1(\alpha,-\alpha)<t_1(-\alpha,\alpha)<t_2(\alpha,-\alpha)<t_2(-\alpha,\alpha)\\
		&<\cdots<t_j(\alpha,-\alpha)<t_j(-\alpha,\alpha)<\cdots<t_m(\alpha,-\alpha)<t_m(-\alpha,\alpha)<1.
	\end{aligned}
\end{equation*}
\end{proposition}

\begin{proof} 

	Thanks to \cite[4.1.3]{szeg1939orthogonal} we obtain 
	\begin{equation}\label{tran-J}
		J_m^{\beta,\gamma}(t)=(-1)^mJ_m^{\gamma,\beta}(1-t).
	\end{equation}
	By \cite[4.22.2]{szeg1939orthogonal} we have
	\begin{equation}\label{J1}
		J_m^{-1,1}(t)=\frac{m+1}{m}(t-1)J_{m-1}^{1,1}(t).
	\end{equation}
	Utilizing \eqref{tran-J} and \eqref{J1} we have
	\begin{equation*}
		\begin{aligned}
			J_m^{1,-1}(t)=(-1)^mJ_m^{-1,1}(1-t)&=\frac{m+1}{m}(-1)^m(-t)J_{m-1}^{1,1}(1-t)\\
			&=\frac{m+1}{m}(-1)^m(-t)(-1)^{m-1}J_{m-1}^{1,1}(t)
		\end{aligned}
	\end{equation*}
	which yields
	\begin{equation}\label{J2}
		J_m^{1,-1}(t)=\frac{m+1}{m}tJ_{m-1}^{1,1}(t).
	\end{equation}
	A. Markov's theorem (see, e.g., Theorem 6.21.1 of  
	\cite{szeg1939orthogonal}) immediately implies that
	$$\frac {\partial t_j(\beta,\gamma)}{\partial \beta} <0 \quad \hbox{ and }
	\quad 
	\frac {\partial t_j(\beta,\gamma)}{\partial \gamma} >0, \quad \mbox{for }\beta,\gamma>-1.$$
	So it follows that 
	\begin{equation}\label{taa}
		t_j(\alpha,-\alpha)<t_j(\alpha,\alpha)<t_j(-\alpha,\alpha)
	\end{equation}
	and for $\beta,\gamma\in (-1,1)$ with $\beta<\gamma$,
	\begin{equation}\label{mark}
		t_j(\beta,-\beta) > t_j(\gamma,-\gamma), \quad t_j(-\beta,\beta)< t_j(-\gamma,\gamma).
	\end{equation}
	Let $y_1 <y_2 <\cdots,<y_{m-1} $ represent the roots of $J_{m-1}^{1,1}(t)$ and denote $y_0=0$, $y_m=1$. Then 	appealing to \eqref{J1} and \eqref{J2} it follows
	\begin{equation}\label{tv1}
		t_j(\alpha,-\alpha)> \lim_{\alpha\rightarrow 1} 
		t_j(\alpha,-\alpha):=t_j(1,-1)=y_{j-1}
	\end{equation}
	and
	\begin{equation}\label{tv2}
		t_j(-\alpha,\alpha)< \lim_{\alpha\rightarrow 1} 
		t_j(-\alpha,\alpha):=t_j(-1,1)=y_{j}.
	\end{equation}
	Thus by \eqref{taa}, \eqref{tv1} and \eqref{tv2} we obtain
	\begin{equation*}
		\begin{aligned}
			0&<t_1(\alpha,-\alpha)<t_1(-\alpha,\alpha)<y_1<t_2(\alpha,-\alpha)<t_2(-\alpha,\alpha)<y_2\\
			&<\cdots<t_j(\alpha,-\alpha)<t_j(-\alpha,\alpha)<y_j<\cdots<t_m(\alpha,-\alpha)<t_m(-\alpha,\alpha)<1
		\end{aligned}
	\end{equation*}
which ends our proof.
\end{proof}

\begin{proposition} The rational function
	$r_m(t)$ is monotonically decreasing for $t\in [0,\infty)$.
\end{proposition}
\begin{proof} 
	It follows by the proposition above that each term in the product of  \eqref{formulation-rm} is monotonically decreasing and so 	must the product.
\end{proof}

\begin{theorem}[Approximation Property]\label{Pro-est}
	Let $r_m(t)$ denotes the Pad\'e approximation of
	$(1+t)^{-\alpha}$, then we have
	$1=r_0(t)>r_1(t)>r_2(t)>\cdots>r_m(t)>\cdots>(1+t)^{-\alpha}$ for $t \in [0,\infty)$. Furthermore,
	for $t\in [0,1],$
	\begin{equation}\label{main-1}
		r_m(t)-(1+t)^{-\alpha}\le   {c}'_{\alpha}\frac{2^{-4m}t^{2m+1}}{2^{mt}}
	\end{equation}
	where $c'_\alpha\approx\frac{\alpha\pi}{2\Gamma(1-\alpha)\Gamma(1+\alpha)}$  for $m$ big enough.
\end{theorem}

\begin{proof}
	Let $n$ be a 
	non-negative integer and set $z_n=r_n-r_{n+1}$.
	Note that $r_n(t)$ matches the first $2n+1$ terms of Maclaurin expansion of
	$(1+t)^{-\alpha}$ while $r_{n+1}(t)$ matches two more, so the Maclaurin expansion of $z_n(t)$ should start from $t^{2n+1}$.  This implies that corresponding expansion of $Q_n(t)Q_{n+1}(t)
	z_n(t)$ must start from $t^{2n+1}$. On the other hand, note that
	\begin{equation*}
		z_n(t)=\frac{P_n(t)Q_{n+1}(t)-P_{n+1}(t) Q_n(t)}{Q_n(t)Q_{n+1}(t)},
	\end{equation*}
	so
	$$Q_n(t)Q_{n+1}(t)
	z_n(t)=P_n(t)Q_{n+1}(t)-P_{n+1}(t) Q_n(t)\in\cP^{2n+1}.$$
	These two relations imply
	\begin{equation}\label{eta}
		P_n(t)Q_{n+1}(t)-P_{n+1}(t) Q_n(t) = \eta_n t^{2n+1}
	\end{equation}
	with $\eta_n$ denoting a constant depending on $n$ only.
	
	Applying Lemma \ref{lem-PQ} we can write down $P_n(t)$ and $Q_n(t)$ as
	\begin{equation}\label{num}
		\begin{aligned}
			P_n(t)&=(-t)^{n} J_n^{\alpha,-\alpha}(-1/t)={}_2F_1(-n,\alpha-n; -2n;-t)\\
			&=\sum_{j=0}^n \frac {(-n)_j
				(\alpha-n)_j} {j! (-2n)_j} (-t)^j 
			:= 1 +\sum_{j=1}^n  a_n^j b_n^j(-\alpha) t^j
		\end{aligned}
	\end{equation}
	and
	\begin{equation}\label{den}
		\begin{aligned}
			Q_n(t)&=(-t)^{n} J_n^{-\alpha,\alpha}(-1/t)
			= \Fto(-n,-\alpha-n; -2n;-t)\\
			&=\sum_{j=0}^n \frac
			{(-n)_j\, (-\alpha-n)_j}
			{j! (-2n)_j}  (-t)^j 
			:= 1 +\sum_{j=1}^n  a_n^j b_n^j(\alpha) t^j
		\end{aligned}
	\end{equation}
	where
	\begin{equation}\label{eq4b}
		b_n^j(\alpha)= (n+\alpha)((n-1)+\alpha)\cdots((n+1-j)+\alpha)
	\end{equation}
	and
	\begin{equation}\label{eq4a}
		a_n^j=\frac {n(n-1)\cdots(n+1-j)}{j! 2n(2n-1)\cdots(2n+1-j)}\qquad
		\hbox{for }j=1,2,\ldots,n.
	\end{equation}
	The relation   \eqref{eq4b} implies that $b_{n+1}^{n+1}(\alpha)=(n+1+\alpha) b_n^n(\alpha)$
	so that
	%and \eqref{den} gives 
	\begin{equation*}
		\eta_n =a_n^nb_n^n(-\alpha)a_{n+1}^{n+1}b_{n+1}^{n+1}(\alpha)-a_n^nb_n^n(\alpha)a_{n+1}^{n+1}b_{n+1}^{n+1}(-\alpha)=
		2\alpha a^n_na_{n+1}^{n+1}b_n^n({\alpha})b_n^n(-\alpha).
	\end{equation*}
	Further applying \eqref{eq4a} we get	
	$$
	\eta_n=c_\alpha\frac{\Gamma(n+1-\alpha)\Gamma(n+1+\alpha)\Gamma(n+1)\Gamma(n+2)}{\Gamma(2n+1)\Gamma(2n+3)}
	$$
	where
	$c_\alpha=\frac{2\alpha}{\Gamma(1-\alpha)\Gamma(1+\alpha)}$. Thus, we get $\eta_n>0$ which implies that $r_{n}(t)>r_{n+1}(t)$. Furthermore  appealing to  Stirling's formula \cite[formula (6.1.39)]{MI1965}: for $t\gg 1$
	\begin{equation*}
		\Gamma(t+b)\approx \sqrt{2\pi}e^{-t} t^{t+b-1/2} 
	\end{equation*}
	it follows 
	
	\begin{equation*}
		\frac{\Gamma(n+1-\alpha)\Gamma(n+1+\alpha)\Gamma(n+1)\Gamma(n+2)}{\Gamma(2n+1)\Gamma(2n+3)}\approx {2\pi}2^{-4n-3}=\frac{\pi}{4} 2^{-4n}. 
	\end{equation*}
	Thus we get $\eta_n\le c'_\alpha 2^{-4n}$ with $c'_\alpha\approx \frac{\pi}{4} c_\alpha$.        
	By telescoping sums, it is clear that 
	\begin{equation}
		r_m(t)-(1+t)^{-\alpha}=\sum_{n=m}^{\infty}(r_{n}-r_{n+1})\le
		c'_\alpha \sum_{n=m}^{\infty}
		\frac{2^{-4n} t^{2n+1}}{Q_n(t)Q_{n+1}(t)}.
		\label{teles}
	\end{equation}
	So to get the desired bound \eqref{main-1} it remains to provide a lower bound for the denominator. % from    below.
	
	Let $t_j(\beta,\gamma)$ be as claimed before and set $\xi_j=t_{n+1-j}(0,0)$, for
	$j=1,\ldots,n$  to be the
	roots of the $n$'th Legendre polynomial $J_n^{0,0}(t)$.
	Since $Q_n(0)=1$,
	\begin{equation*}
		Q_n(t)=\prod_{j=1}^{n}(1+t_j(-\alpha,\alpha) t).
	\end{equation*} 
	As in \eqref{mark},
	$$t_j(-\alpha,\alpha)>t_j(0,0)$$
	so that
	$$Q_n(t)\ge \prod_{i=1}^{n}(1+\xi_i  t):=Q_n^0(t) ,   \ \ \mbox{for all } \ \ t \ge 0.
	$$
	Now,
	\begin{equation*}  
		\log_2 Q^0_n(t)=\sum_{i=1}^{n}\log_2(1+{\xi}_i t).
	\end{equation*} 
	Recall that here $\xi_i$, $i=1, \cdots, n$ are the roots of Legendre polynomial. Applying the inequalities for the roots of Legendre polynomials, see, e.g. \cite[ Equation (2)]{szego1936inequalities} %, we obtain
	\begin{equation*}
		\cos\left(\frac{2i}{2n+1}\pi\right)<2 {\xi}_i-1<\cos\left(\frac{2i-1}{2n+1}\pi\right),\quad i=1,2,\cdots,n,
	\end{equation*}
	we get  
	\begin{equation*}
		\begin{aligned}
			\log_2 Q_n(t)&>\sum_{i=1}^{n}\log_2\left(1+\frac{t}{2}\left(1+\cos\frac{2i\pi}{2n+1}\right)\right).
		\end{aligned}
	\end{equation*}
	Note that $\frac{t}{2}(1+\cos\frac{2i\pi}{2n+1}) \in [0,1]$  and $\log_2(1+\eta)\ge \eta$ for $\eta\in [0,1]$ so 
	\begin{equation*}
		\begin{aligned}
			\log_2 Q_n(t)&>\sum_{i=1}^{n}\frac{t}{2}\left(1+\cos\frac{2i\pi}{2n+1}\right)=\frac{nt}{2}+\frac{t}{2}\sum_{i=1}^{n}  \cos\frac{2i\pi}{2n+1}.
		\end{aligned}
	\end{equation*}
	The well known identity $\sum_{i=1}^{n}\cos\frac{2i\pi}{2n+1}=- \frac12$ implies $\log_2 Q_n(t)>\frac{nt}{2}-\frac{t}{4}$ so that
	\begin{equation*}
		Q_n(t)Q_{n+1}(t)\ge 2^{\frac{nt}{2}-\frac{t}{4}}2^{\frac{(n+1)t}{2}-\frac{t}{4}}= 2^{nt}.
	\end{equation*}
	Combining this with \eqref{teles} gives
	\begin{equation*}
		r_m(t)-(1+t)^{-\alpha}\le
		\frac{32}{31}{c}'_{\alpha}\frac{2^{-4m}t^{2m+1}}{2^{mt}}    \quad     \hbox{for }t\in [0,1]
	\end{equation*}
	with
	\begin{equation*}
		c'_\alpha=\frac{\pi}{4} c_\alpha\approx\frac{\alpha\pi}{2\Gamma(1-\alpha)\Gamma(1+\alpha)},
	\end{equation*}
	which ends our proof by ignoring the factor $\frac{32}{31}$.
\end{proof}

Note that the estimate of Theorem \ref{Pro-est} is crucial in the the theoretical analysis of the method. Thus, the question how robust and sharp is the estimate \eqref{main-1} is of great importance. In Fig.\ref{figure_test-1} we  show the actual and the predicted by  \eqref{main-1} errors for various values of $t$, $\alpha$, and $m$. All graphs show that 
actual error and predicted errors differ insignificantly, with the biggest difference, as one expects, occurring at $t=1$.

\begin{figure}
	\centering
	\begin{tabular}{c c c}
		\subfigure[$\alpha=0.1$]
		{\includegraphics[width=0.33\textwidth,trim=140 210 160 220, clip]{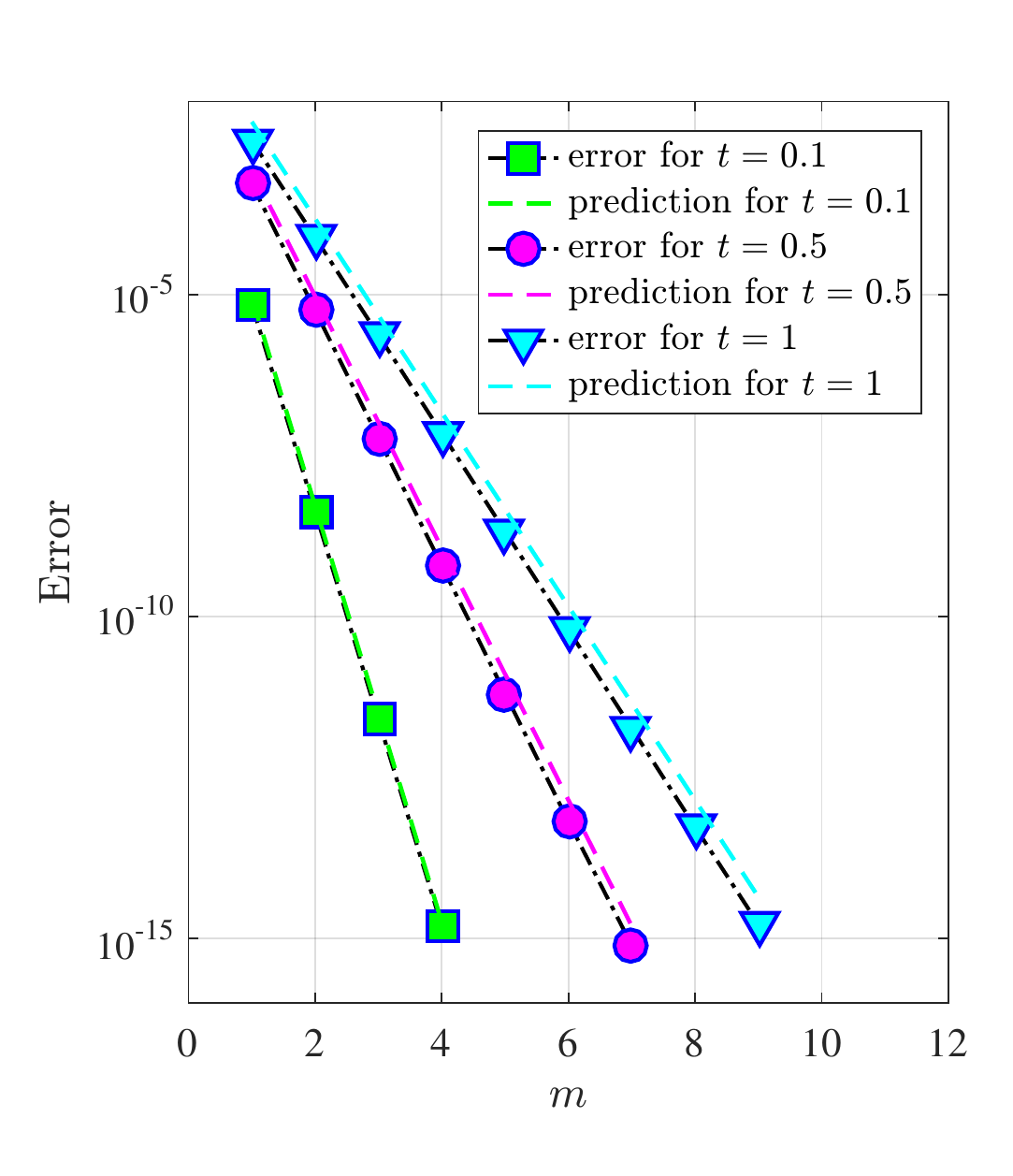}}
		\subfigure[$\alpha=0.5$]
		{\includegraphics[width=0.33\textwidth,trim=140 210 160 220, clip]{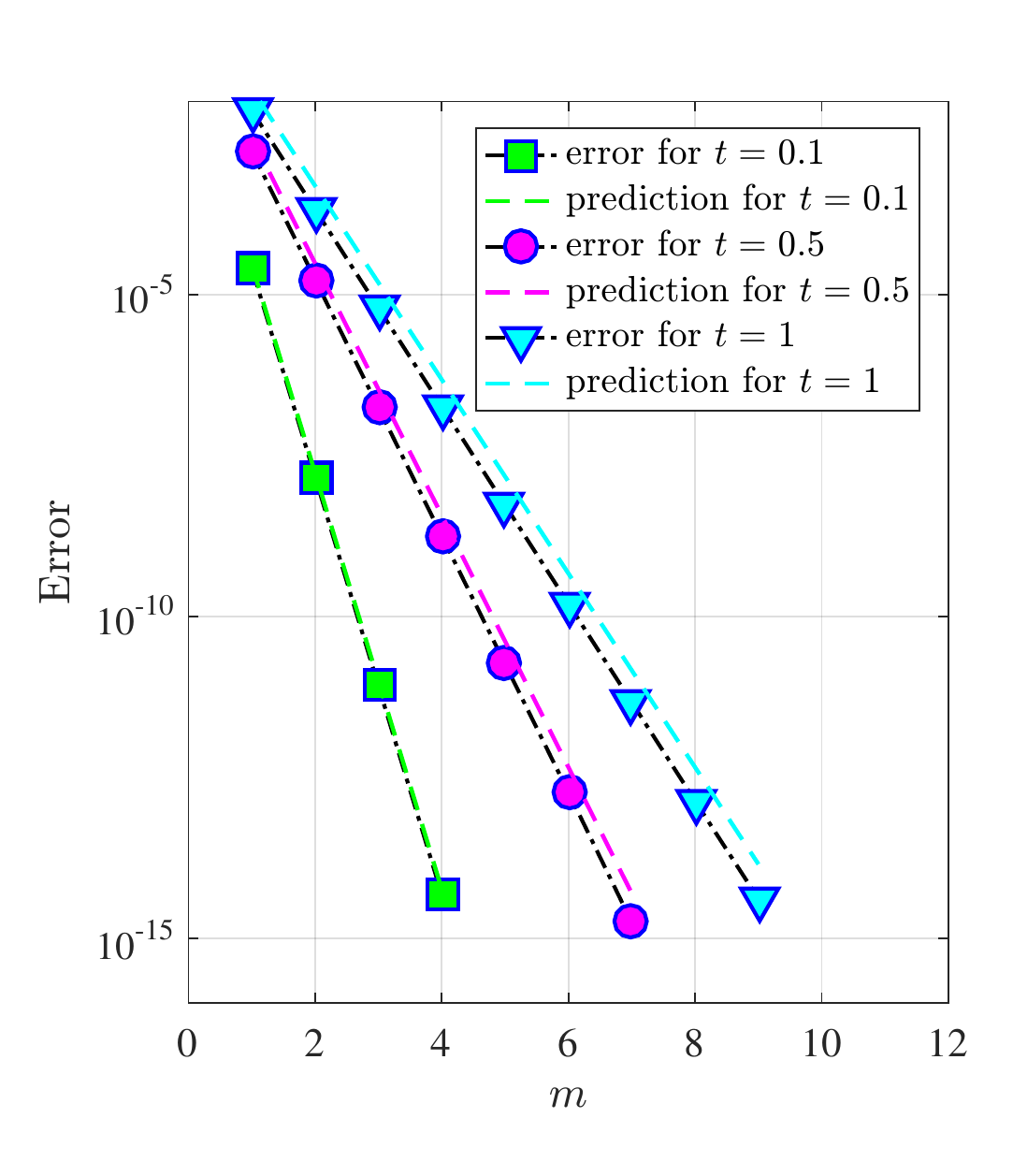}}
		\subfigure[$\alpha=0.9$]
		{\includegraphics[width=0.33\textwidth,trim=140 210 160 220, clip]{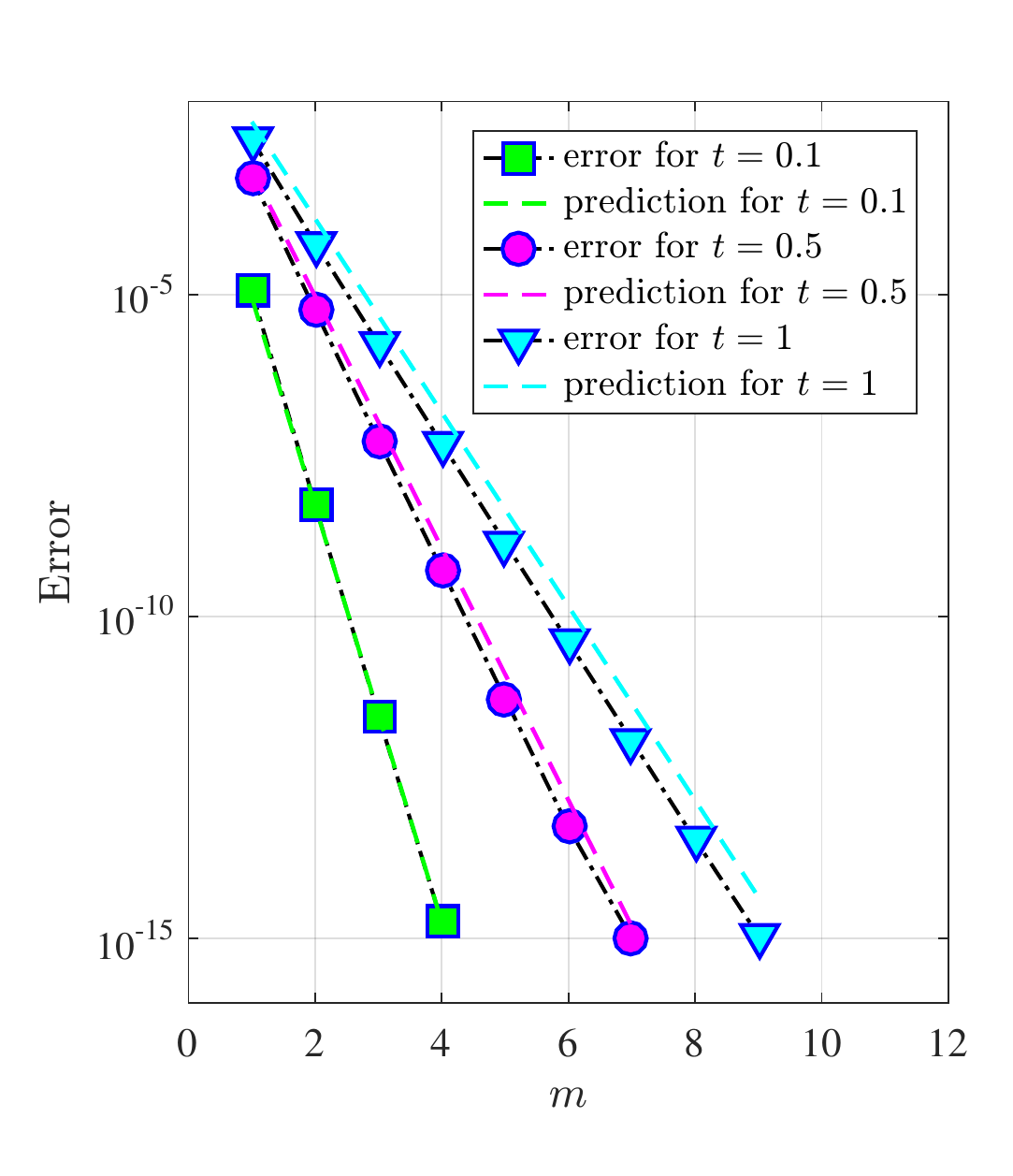}}
	\end{tabular}
	\caption{$r_m(t)-(1+t)^{-\alpha}$ for different $\alpha$ and $t$,  and theoretical predictions by Theorem \ref{Pro-est} involving $c'_\alpha$. }
	\label{figure_test-1}
\end{figure}
%%%%%%%%%%%%%%%%%%%%%%%%%%%%%%%%%%%%%%%%%%%
\section{Approximation for $\cLgh^{-\alpha}$ and error estimate }\label{Sec-3}
Suppose $f_h\in V_h$ is an appropriate approximation of $f$, then we approximate $u$ by $u_h=\cLgh^{-\alpha}f_h$. Applying \eqref{prod-L} we can further evaluate $u_h$ by the following approximant
\begin{equation}\label{appro-prod-L}
	u_h\approx U_{L+1}=\frac{1}{\tlambda^\alpha}\prod_{k=0}^{L}r_m(\tau_k\cBh(\tlambda\cI+t_k\cBh)^{-1}).
\end{equation}

In this section, we shall analyze the error between $u_h$ and $U_{L+1}$.  Recall that the spectra of $\cLgh$ is  contained in $[\lambda_{h,min},\lambda_{h,max}]$,  
$\tlambda\in (0,\lambda_{h,min}]$,  $\cBh=\cLgh-\tlambda\cI$, 
$t_k\in [0,1]$ with $t_0=0$ and $t_{L+1}=1$, and $\tau_k=t_{k+1}-t_k$,
$k=0,1,\dots,L$. 

Thanks to \eqref{appro-prod-L}, the approximation scheme can be obtained from the recurrence
\begin{equation}\label{Scheme-1}
\begin{aligned}
	U_0=&\tlambda^{-\alpha} f_h  \\
	U_n= & r_m(\tau_{n-1} \cBh (\tlambda \cI +t_{n-1} \cBh)^{-1}) U_{n-1},\   \hbox{ for } n=1,2,\ldots,L+1,\\
	\ U_{L+1} & \quad \mbox{is an approximation to } \ \ u_h=\cLgh^{-\alpha} f_h.
\end{aligned}
\end{equation}

If we 
%expand $U_{L+1}$ in 
use the expansion
$
U_{L+1}=\sum_{j=1}^D \mu(\lambda_{j,h})  (f_h,\psi_{j,h})_{\cM_h} \psi_{j,h},
$
then \eqref{Scheme-1} implies that the coefficient $\mu(\lambda)$ satisfies the following recurrence:
\begin{equation}\label{Scheme-lam}
\begin{aligned}
	\mu_0(\lambda)&=\tlambda^{-\alpha},\\
	\mu_n(\lambda)&= r_m(\tau_{n-1}(\lambda-\tlambda) (\tlambda   +t_{n-1}(\lambda-\tlambda))^{-1})\mu_{n-1}(\lambda),\quad \hbox{ for }
	n=1,2,\ldots,L+1,
\end{aligned}
\end{equation}
and set $\mu(\lambda)=\mu_{L+1}(\lambda)$, with $\lambda=\lambda_{j,h}$, $j=1, \dots, D$.

\begin{theorem}\label{err-t}
	Denote $\theta_n(\lambda)=\frac{\tau_n(\lambda-\tlambda)}{\tlambda+t_n(\lambda-\tlambda)}$ for $n=0,1,\cdots,L$. Suppose the mesh $\{t_n\}_{n=0}^{L+1}$ satisfies $\theta_n(\lambda_{h,max})\le 1$ and  $\frac{\tau_{n+1}}{\tau_n}\le \nu$. 
	Then for positive integer $m$ we have
	\begin{equation}\label{t1}
		|\lambda^{-\alpha} - \mu(\lambda)| \le\widetilde{c}\tlambda^{-\alpha}{2^{-5m}} \quad \forall
		\lambda\in [\lambda_{h,min},\lambda_{h,max}]
	\end{equation}
	and 
	\begin{equation}\label{t2}
		\|\cLgh^{-\alpha} f_h -U_{L +1}\|_{L^2(\cM_h)} \le \widetilde{c}\tlambda^{-\alpha}{2^{-5m}} \| f_h \|_{L^2(\cM_h)},
	\end{equation}
	where $\widetilde{c}=(\alpha+\nu)2^\alpha \frac{c'_\alpha}{\alpha}$ with $c'_\alpha\approx\frac{\alpha\pi}{2\Gamma(1-\alpha)\Gamma(1+\alpha)}$ given in Theorem \ref{Pro-est}. 
\end{theorem}

\begin{proof}
	Fix $\lambda$ in  $[\lambda_{h,min},\lambda_{h,max}]$ and let
	$\mu_n$ be as \eqref{Scheme-lam}. Let
	$v_{n}(\lambda)=(\tlambda +t_{n}(\lambda-\tlambda))^{-\alpha}$,
	$e_{n}=\mu_{n}-v_{n}$  for $n=0,\ldots,L+1$. Furthermore, as claimed we denote
	$$\theta_{n}(\lambda)= \frac{\tau_n (\lambda-\tlambda)}{\tlambda+t_{n}(\lambda -\tlambda)},
	$$
	for $n=0,\ldots,L$. We note that, 
	\begin{equation*}
		v_{n+1}=\left(1+\theta_{n}\right)^{-\alpha}v_{n}.
	\end{equation*}
	Thus,
	$$
	e_{n+1} =(1+\theta_{n})^{-\alpha}
	e_{n} + [r_m(\theta_{n})-(1+\theta_{n})^{-\alpha}] \mu_{n}.$$
	Note $r_m(t)\le 1$ and thus $\mu_n\le \mu_0=\tlambda^{-\alpha}$. Note that $\theta_n(\lambda_{h,max})\le 1$ implies $\theta_n(\lambda)\le 1$ for any $\lambda\in[\tlambda,\lambda_{h,max}]$, so applying Theorem  \ref{Pro-est}  we get
	\begin{equation*}
		|e_{n+1}| \le (1+\theta_{n})^{-\alpha}
		|e_{n}| +  {{c}'_{\alpha}}{\tlambda^{-\alpha}}\frac{2^{-4m}{\theta_n}^{2m+1}}{2^{m\theta_n}}
	\end{equation*}
	for $n=0,1,\ldots,L$. Note that $g(t)=t^{2m}{2^{-mt}}$ increase with $t$ for $t\in [0,1]$ and $g(1)=2^{-m}$ so we have
	\begin{equation*}
		|e_{n+1}| \le (1+\theta_{n})^{-\alpha}
		|e_{n}| +  {{c}'_{\alpha}}{\tlambda^{-\alpha}} 2^{-5m}\theta_n
	\end{equation*}
	We shall derive the error by recursion. For notational simplicity, we denote $\eta={{c}'_{\alpha}}{\tlambda^{-\alpha}} 2^{-5m}$, then
	\begin{equation*}
		\begin{aligned}
			|e_{L+1}|&\le (1+\theta_{L})^{-\alpha}|e_{L}|+\eta\theta_L\\
			&\le (1+\theta_{L})^{-\alpha}(1+\theta_{L-1})^{-\alpha}|e_{L-1}|+(1+\theta_{L})^{-\alpha}\eta\theta_{L-1}+\eta\theta_L\\
			&\le \cdots\\
			&\le \eta\theta_L+(1+\theta_{L})^{-\alpha}\eta\theta_{L-1}+\cdots+\prod_{i=2}^{L}(1+\theta_{i})^{-\alpha}\eta\theta_1+\prod_{i=1}^{L}(1+\theta_{i})^{-\alpha}\eta\theta_0
		\end{aligned}
	\end{equation*}
Use the fact that $\theta_L\le 2^\alpha(1+\theta_L)^{-\alpha}\theta_L$ and for $k=1,2,\cdots, L$
\begin{equation*}
	\prod_{i=k}^{L}(1+\theta_{i})^{-\alpha}\theta_{k-1}\le 2^\alpha \prod_{i=k-1}^{L}(1+\theta_{i})^{-\alpha}\theta_{k-1}
\end{equation*}
then we obtain
\begin{equation*}
	|e_{L+1}|\le (1+\theta_{L})^{-\alpha} \eta'\theta_L+\prod_{i=L-1}^{L}(1+\theta_{i})^{-\alpha}\eta'\theta_{L-1}+\cdots+\prod_{i=1}^{L}(1+\theta_{i})^{-\alpha}\eta'\theta_1+\prod_{i=0}^{L}(1+\theta_{i})^{-\alpha}\eta'\theta_0.
\end{equation*}
	with $\eta'=2^{\alpha}\eta$. Note that $\tau_n=t_{n+1}-t_n$ and $t_{L+1}=1$, so for $n\le L$
	\begin{equation*}
		\prod_{i=n}^{L}(1+\theta_{i})^{-\alpha}=\prod_{i=n}^{L}\left[\frac{\tlambda+t_{i+1}(\lambda-\tlambda)}{\tlambda+t_i(\lambda-\tlambda)}\right]^{-\alpha}=\left[\frac{\tlambda+t_{L+1}(\lambda-\tlambda)}{\tlambda+t_n(\lambda-\tlambda)}\right]^{-\alpha}=\frac{\lambda^{-\alpha}}{(\tlambda+t_n(\lambda-\tlambda))^{-\alpha}}.
	\end{equation*}
	Thus 
	\begin{equation*}
		\begin{aligned}
			|e_{L+1}|\le \eta'\sum_{n=0}^{L}\frac{\lambda^{-\alpha}\theta_n}{(\tlambda+t_n(\lambda-\tlambda))^{-\alpha}}&=\frac{\eta'}{\lambda^{\alpha}}\sum_{n=0}^{L}\frac{\tau_n(\lambda-\tlambda)}{(\tlambda+t_n(\lambda-\tlambda))^{1-\alpha}}\\
			&=  \frac{\eta'}{\lambda^{\alpha}} \frac{\tau_0(\lambda-\tlambda)}{\tlambda^{1-\alpha}}+\frac{\eta'}{\lambda^{\alpha}}\sum_{n=0}^{L-1}\frac{\tau_{n+1}(\lambda-\tlambda)}{(\tlambda+t_{n+1}(\lambda-\tlambda))^{1-\alpha}}\\
			&\le \eta'\frac{\tlambda^{\alpha}}{\lambda^{\alpha}} \theta_0+ \nu\frac{\eta'}{\lambda^{\alpha}}\sum_{n=0}^{L-1}\frac{\tau_{n}(\lambda-\tlambda)}{(\tlambda+t_{n+1}(\lambda-\tlambda))^{1-\alpha}}\\
%			&\le \eta'+\nu\frac{\eta'}{\lambda^{\alpha}}\int_{0}^{1}\frac{\lambda-\tlambda}{(\tlambda+t(\lambda-\tlambda))^{1-\alpha}}dt\\
%			&=\eta'+\nu\frac{\eta'}{\lambda^{\alpha}}\int_{0}^{\lambda-\tlambda}\frac{dy}{(\tlambda+y)^{1-\alpha}}\le (1+\nu)\eta'
		\end{aligned}
	\end{equation*}
Since $\theta_0\le 1$, $\tlambda\le \lambda$ and  $\frac{\lambda-\tlambda}{(\tlambda+t(\lambda-\tlambda))^{1-\alpha}}$ is monotonically decreasing, it follows
\begin{equation*}
		\begin{aligned}
		|e_{L+1}|
		&\le \eta'+\nu\frac{\eta'}{\lambda^{\alpha}}\int_{0}^{1}\frac{\lambda-\tlambda}{(\tlambda+t(\lambda-\tlambda))^{1-\alpha}}dt\le \left(1+\frac{\nu}{\alpha}\right)\eta'
	\end{aligned}
\end{equation*}
	which leads to 
	\begin{equation*}
		\begin{aligned}
			|\lambda^{-\alpha} - \mu(\lambda)| =|e_{L+1}|\le \widetilde{c}\tlambda^{-\alpha}{2^{-5m}}.
		\end{aligned}
	\end{equation*}
	where $\widetilde{c}=(\alpha+\nu)2^\alpha \frac{c'_\alpha}{\alpha}$ with $c'_\alpha\approx\frac{\alpha\pi}{2\Gamma(1-\alpha)\Gamma(1+\alpha)}$ given in Theorem \ref{Pro-est}. 
	
	By Parseval's identities it follows:
\begin{equation*}
\begin{aligned}
		\| \cLgh^{-\alpha} f_h -U_{L+1}\|_{L^2(\cM_h)}&=\left (\sum_{j=1}^D
	|\lambda_{j,h}^{-\alpha}-\mu(\lambda_{j,h})|^2 |( f_h,\psi_{j,h})_{\cM_h}|^2 \right )^\frac12 \\
	&\le \widetilde c \tlambda^{-\alpha}{2^{-5m}}\| f_h \|_{L^2(\cM_h)}
\end{aligned}
\end{equation*}
	and this completes the proof.
\end{proof}
Obviously the geometrically graded mesh $t_{n+1}=2^nt_1$ with proper $t_1\le \lambda^{-1}_{h,max}$ 
such that $t_{L+1}=1$ in \cite{DLP-Pade} satisfies $\theta_n(\lambda_{h,max})\le 1$. However, to reach $t_{L+1}=1$ with fewer steps, say, to get smaller $L$, the best choice would be $\theta_{n}(\lambda_{h,max})=1$ for $n=0,1,\cdots,L$. Thus we obtain $t_0=0$, 
\begin{equation}\label{mesh_t}
	t_1=\frac{\tlambda}{\lambda_{h,max}-\tlambda},
	\ \ t_{n+1}=\min\{2^{n+1}t_1-t_1,1\}, \quad n=1,2,\cdots,L,
\end{equation}
thus
\begin{equation}\label{num_L}
	 L+1=\lceil \log_2 (\lambda_{h,max}/\tlambda)\rceil.
\end{equation}

Fig.\ref{Fig-ans-r} presents the error $e(\lambda)=|\mu(\lambda)-\lambda^{-\alpha}|$ for $\lambda\in[2,2^{50}]$ under various $\alpha$ and $m$, where we took $\hat{\lambda}=1$ in the computation. In fact, if one considers the relative error $\lambda^{\alpha}{e(\lambda)}$, then $m=10$ can give an approximation with $\mathcal{O}(10^{-14})$ accuracy for $\alpha=0.1,0.5$ and $0.9$. 
\begin{figure}[!htp]
	\centering
	\begin{tabular}{ccc}
		\subfigure[$m=4$]
		{\includegraphics[width=0.33\textwidth,trim=80 230 100 220, clip]{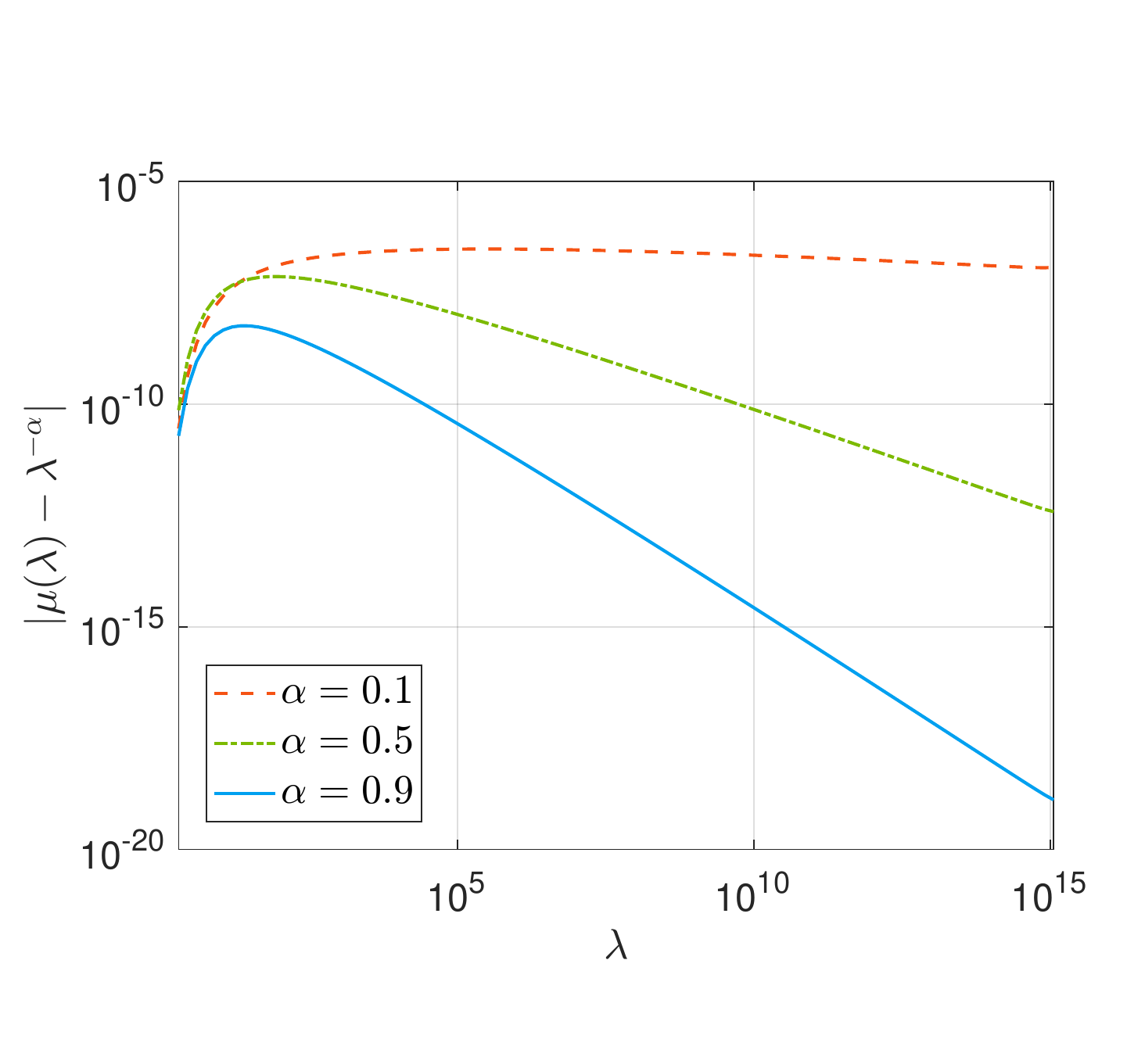}}
		\subfigure[$m=7$]
		{\includegraphics[width=0.33\textwidth,trim=80 230 100 220, clip]{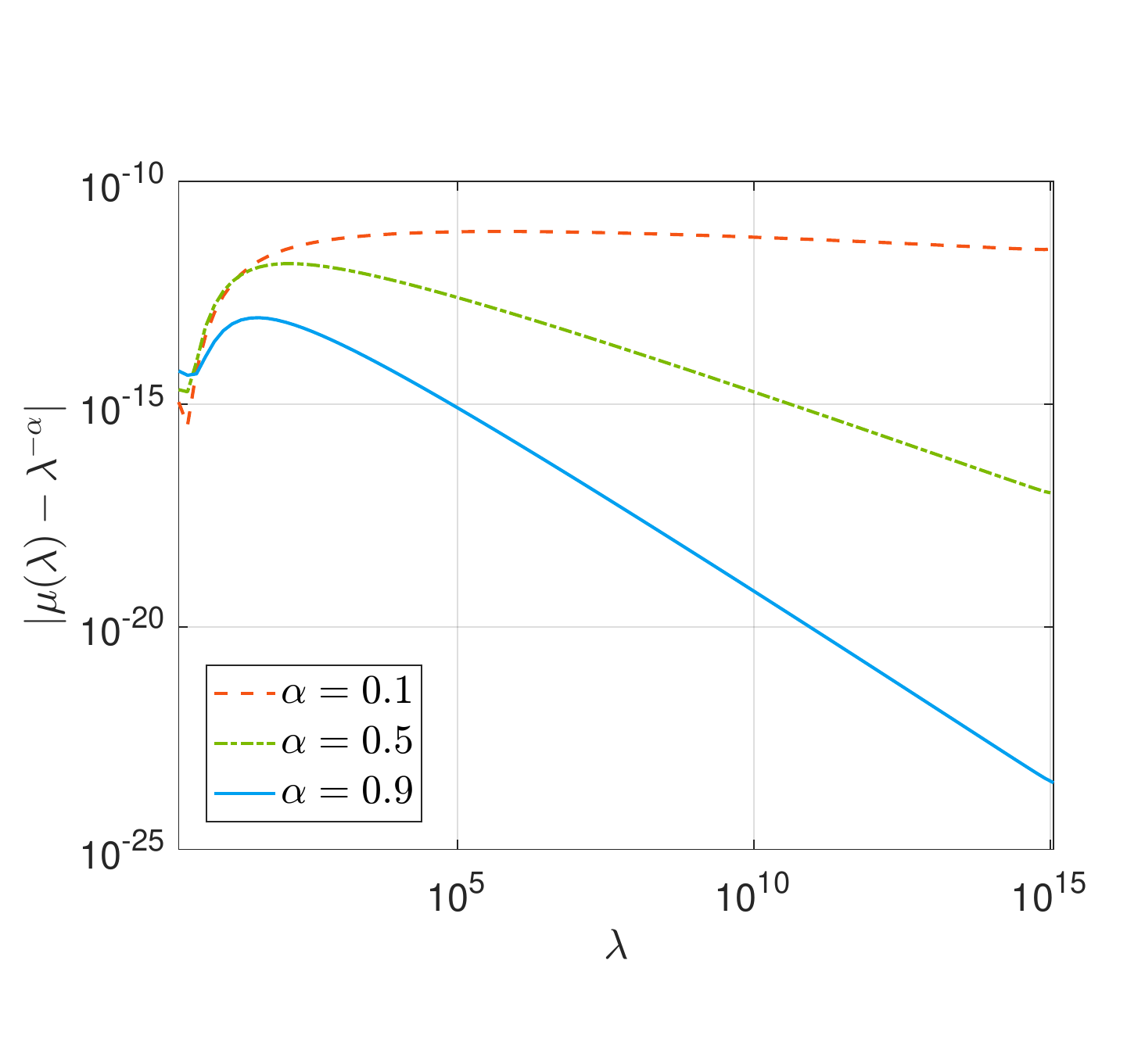}}
		\subfigure[$m=10$]
		{\includegraphics[width=0.33\textwidth,trim=80 230 100 220, clip]{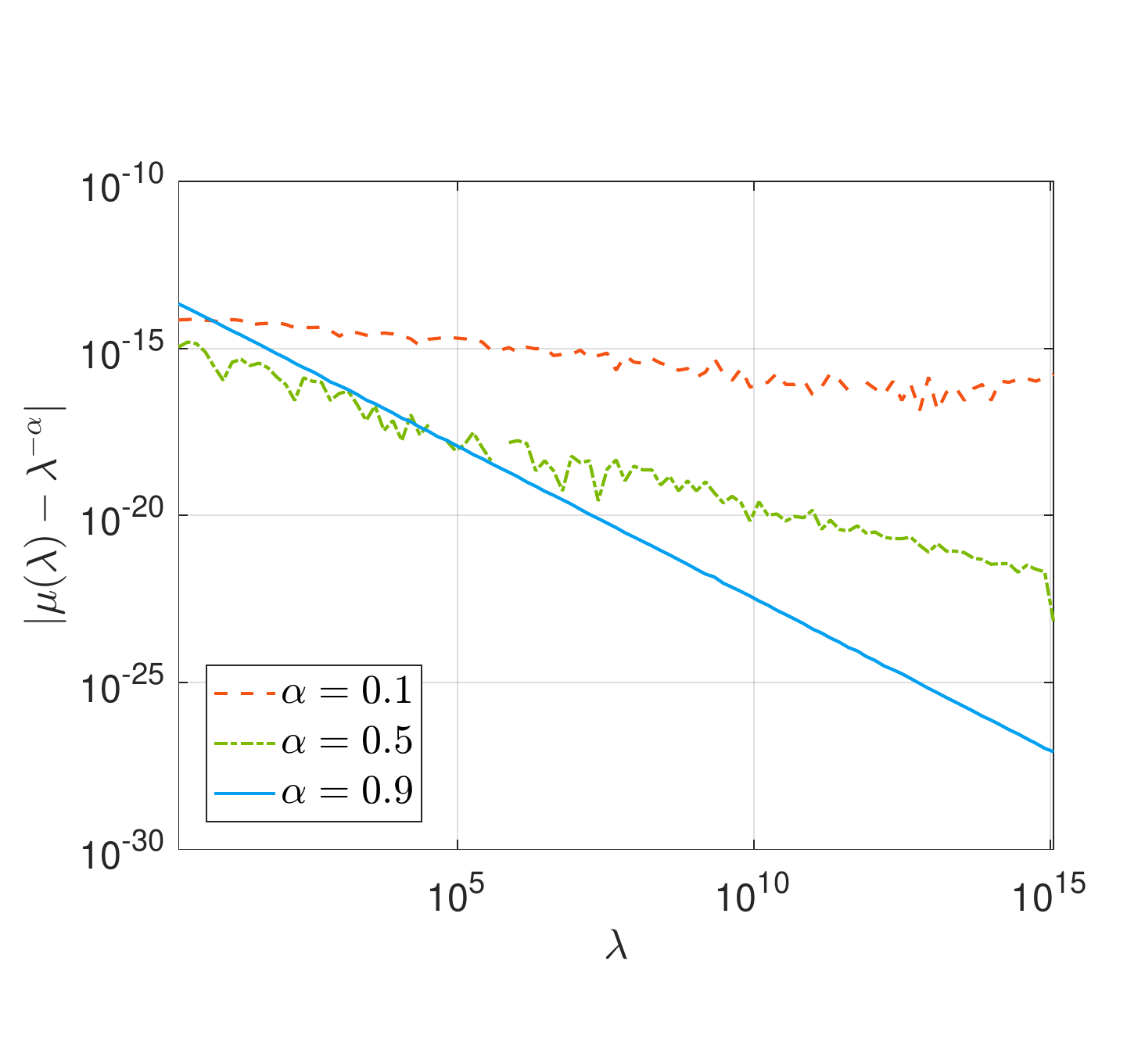}}
	\end{tabular}
	\caption {Error $ |\mu(\lambda)-\lambda^{-\alpha}|$ with respect to $\lambda$ under different $\alpha$ and $m$.}
	\label{Fig-ans-r}
\end{figure}

Note that the mesh constructed in \eqref{mesh_t} yields  $\frac{\tau_{n+1}}{\tau_n}\le 2$, so by setting $\nu=2$ in Theorem \ref{err-t} we arrive at the following corollary.
\begin{corollary}\label{err-t-cor}
	Suppose the mesh $\{t_i\}_{i=0}^{L+1}$ is constructed by \eqref{mesh_t} and denote by $N_s=m(L+1)$ the number of total solves, then
	\begin{equation}\label{t3}
		\|\cLgh^{-\alpha} f_h -U_{L +1}\|_{L^2(\cM_h)} \le \hat{c}\tlambda^{-\alpha}{32^{-\frac{N_s}{\lceil \log_2 (\lambda_{h,max}/\tlambda)\rceil}}} \| f_h \|_{L^2(\cM_h)},
	\end{equation}
	where $\hat{c}\approx \frac{(\alpha+2)2^{\alpha-1}\pi}{\Gamma(1-\alpha)\Gamma(1+\alpha)}$. 
\end{corollary}

Now we shall show the stability of the algorithm. Recall that
\begin{equation*}
	r_m(t)=\prod_{i=1}^m\frac{1+t_i(\alpha,-\alpha)t}{1+t_i(-\alpha,\alpha)t}
	\quad \mbox{ and } \quad 
	{U}_{l+1} =r_m\left(\tau_l \cBh (\tlambda \cI+t_l\cBh)^{-1} \right) {U}_{l},
\end{equation*}
so appealing to Proposition \ref{interlace} one can directly obtain all the eigenvalues of the matrix $r_m\left(\tau_l \cBh (\tlambda \cI+t_l\cBh)^{-1} \right)$ lies in the interval $(0,1)$, which implies the stability.

In fact, to implement the algorithm more efficiently, we rewrite  the rational function $r_m(t)$ as a sum of partial fractions:
\begin{equation}\label{partial_frac}
	r_m(t)=\beta_0+\sum_{i=1}^m\frac{\beta_i}{1+t_i(-\alpha,\alpha)t}
\end{equation}
where 
\begin{equation}\label{eq-beta}
	\beta_0=\prod_{i=1}^{m}\frac{t_i(\alpha,-\alpha)}{t_i(-\alpha,\alpha)} >0 
	\quad \mbox{ and }\quad 
	\beta_i=\frac{\prod_{j=1}^{m}(1-t_j(\alpha,-\alpha)/t_i(-\alpha,\alpha))}{\prod_{j\neq i}(1-t_j(-\alpha,\alpha)/t_i(-\alpha,\alpha))} >0.
\end{equation}
Furthermore, using the matrix representation of the operator $\cLgh=\mbM^{-1} \mbS$ 
through the stiffness matrix $\mbS$ and the mass matrix $\mbM$, the operator $r_m(\tau_l\cBh(\tlambda\cI+t_l\cBh)^{-1})$ could be written in matrix form as
\begin{equation}\label{matrix_r}
	\begin{aligned}
		&\beta_0\mbI+\sum_{i=1}^{m}\beta_i[\tlambda\mbM+(t_l+t_i(-\alpha,\alpha)\tau_l)(\mbS-\tlambda\mbM)]^{-1}[\tlambda\mbM+t_l(\mbS-\tlambda\mbM)]\\
		&:=\beta_0\mbI+\sum_{i=1}^m\mbR_l^i.
	\end{aligned}
\end{equation}
In these notation $\mbS \vec{\psi}_{j,h}  = \lambda_{j,h} \mbM \vec{\psi}_{j,h}$ with $\{\vec{\psi}_{j,h}\}_{j=1}^{D}$ the eigenvectors.  The choice $\tlambda \in(0, \lambda_{1,h}]$ 
and the properties of the Pad\'e approximation ensures that 
the  $ \tlambda\mbM+(t_l+t_i(-\alpha,\alpha)\tau_l)(\mbS-\tlambda\mbM) $ is positive definite.

Suppose $v_h \in V_h$ is given as a column vector $\vec{v} =(v_1, \dots, v_D)^T \in \bR^D$ through  representation of $v_h=\sum_{i=1}^D v_i {\phi}_i$ with respect to the nodal basis $\phi_i$ of $V_h$.
Let $\vec{F}$ be the column vector from the representation of $f_h$ with respect to the nodal basis of $V_h$, and let $\vec{U}_l \in \bR^D$ be the vector representation of  the solution function $U_l(\vec{x})$. Then the implementation of the method can be done by the following algorithm:
\begin{algorithm}
	%	\caption{ }
	\label{a-basic}
	\begin{algorithmic}
		\STATE (a) Set $\vec{U}_{0}=\tlambda^{-\alpha}\vec{F}$;
		\STATE(b) {For $l=0,1,\ldots,L:$}
		\INDSTATE[1.5] (i) For $i=1,\cdots,m$, solve for $\vec{U}_{l+1}^i$: \ 
		$
		\vec{U}_{l+1}^i=\mbR_l^i\,\vec{U}_l;
		$
		\INDSTATE[1.5] (ii) Set $\vec{U}_{l+1}=\beta_0 \vec{U}_l+\sum_{i=1}^m \beta_i \vec{U}_{l+1}^i$;
		\STATE (c) end.
	\end{algorithmic}
\end{algorithm}

\noindent It is also worth to point out that step (i) can be done parallelly.
%Moreover, the positivity of the coefficient $\beta_i$ ensures numerical stability of the computations. 
\section{Total Error}\label{Sec-4}
In this section we shall present the error between $\cLg^{-\alpha}f$ and $\cLgh^{-\alpha}f_h$ and then the total error follows.  Different from  the Euclidean case,  the error analysis  on manifolds involves geometric error which makes it much complicated. However, it is natural to  believe the error bound for manifolds should consist with the Euclidean case given in \cite[Theorem 4.3]{bonito2015numerical}. Furthermore, it is worth to point out that our parametric FE scheme is different from the one proposed in \cite{bonito2021approximation} and compared with the error bound in Euclidean case, a $\ln h^{-1}$ factor gets involved in their estimate. Hence, new error analysis to derive the optimal  estimate is needed.

To this end, we first introduce some lemmas which will be used frequently in our later analysis. And  to simplify the  notations we denote $\|\cdot\|^2=(\cdot,\cdot)_{\cM}$  and $\|\cdot\|_h^2=(\cdot,\cdot)_{\cM_h}$.  Hereafter, we always use $c$ as a constant independent of $h$ and $\mu$(we will introduce later) which may be different in different places. We further suppose that $\cM_h$ is quasi-uniform and the number of simplices sharing the same vertex is uniformly bounded(see \cite[Fig.1]{bonito2019posteriori} for the illustration of the later assumption).

\begin{lemma}[see Lemma 4.2 in \cite{dziuk2013finite}]\label{equi-le}
	Let $\eta:\cM_h\rightarrow \bR$ with lift $\eta^\ell: \cM\rightarrow \bR$ and let $K$ be any simplex in $\cK$.  Denote $\tilde{K}=\{\theta(\vec{x})| \vec{x}\in K\}$ with $\theta$ the lift function given in \eqref{eq-lift}, then
	\begin{equation*}
		 \|\eta\|_{L^2(K)}\sim \|\eta^\ell\|_{L^2(\tilde{K})}\quad \mbox{and  }\|\ngh\eta\|_{L^2(K)} \sim \|\ng\eta^\ell\|_{L^2(\tilde{K})}
	\end{equation*}
where by ``$\sim$'' we mean the numbers on both sides of it can be bounded by each other up to a constant independent of $h$. 
\end{lemma}
Square both sides of the two equations in Lemma \ref{equi-le} and sum over $\cK$, then by  taking square root we immediately get $\|\eta\|_h \sim \|\eta^\ell\|$ and $\|\ngh\eta\|_h \sim \|\ng\eta^\ell\|$. Interpolating the two results one immediately obtains
\begin{equation}\label{equi-norm}
	\|\eta\|_{H^\gamma(\cM_h)}\sim \|\eta^\ell\|_{H^\gamma(\cM)}\quad \mbox{for }\forall \gamma\in [0,1].
\end{equation}

For any $\eta\in L^2(\cM_h)$ with vanishing mean value, $\int_\cM \eta^\ell d\cM=0$ is not guaranteed.  To close this gap, we define the following lift operation: for $\eta\in L^2(\cM_h)$ with $\int_{\cM_h} \eta d\cM_h=0$
\begin{equation}\label{new-lift}
	\eta^{\nell}=\eta^\ell-\overline{\eta^{\ell}}\quad \mbox{with }\overline{\eta^{\ell}}=\frac{1}{|\cM|}\int_{\cM} \eta^\ell d\cM.
\end{equation}
Obviously the difference between $\eta^{\nell}$ and $\eta^{\ell}$ is the constant $\overline{\eta^{\ell}}$ which can be bounded by $ch^2$. In fact
\begin{equation}\label{error-nlift}
	\begin{aligned}
		\left| \overline{\eta^{\ell}}\right|=\frac{1}{|\cM|}\left| \int_{\cM} \eta^\ell d\cM-\int_{\cM_h} \eta d\cM_h\right|&=\frac{1}{|\cM|}\left| \int_{\cM} \eta^\ell (1-\sqrt{|g|^{-1}}) d\cM\right|\\
		&\le \|1-\sqrt{|g|^{-1}}\|_{L^\infty(\cM)} \|\eta^\ell\|\\
		&\le ch^2\|\eta^\ell\|\sim ch^2\|\eta\|_h
	\end{aligned}
\end{equation}
where $|g|$ is the determinant of the Riemannian metric $g$, and in the last line we utilized $\|1-\sqrt{|g|^{-1}}\|_{L^\infty(\cM)}\le c h^2$(see,  \cite[Lemma 4.1]{dziuk2013finite}). So applying  $	\eta^\ell=\eta^{\nell}+\overline{\eta^{\ell}}$ and \eqref{error-nlift}  one can verify \eqref{equi-norm}  still holds for the lift operation we defined above, say, for $\eta\in H^\gamma(\cM)$ with vanishing mean value
\begin{equation}\label{equi-norm-new}
	\|\eta\|_{H^\gamma(\cM_h)}\sim \|\eta^{\nell}\|_{H^\gamma(\cM)}\quad \mbox{for }\forall \gamma\in [0,1].
\end{equation}

Corresponding to the lift operation we defined above, for $v\in L^2(\cM)$ with $\int_{\cM} vd\cM=0$ we define
\begin{equation}\label{nlift-inv}
	v^{-\nell}=v^{-\ell}-\overline{v^{-\ell}},\quad \mbox{with }\overline{v^{-\ell}}=\frac{1}{|\cM_h|}\int_{\cM_h} v^{-\ell} d\cM_h,
\end{equation}
then one can easily check  $\int_{\cM_h}v^{-\nell}d\cM_h=0$. Furthermore, for $v\in L^2(\cM)$ with vanishing mean value on $\cM$, we have $(v^{-\nell})^{\nell}=v$. In fact, let $\overline{(\cdot)}$ represent the average value of $(\cdot)$ on $\cM$ or $\cM_h$(depending on the domain of $(\cdot)$), then by definition and utilizing the fact that $\overline{v^{-\ell}}$ is a constant it follows
\begin{equation}\label{inverse-nlift}
\begin{aligned}
	(v^{-\nell})^{\nell}=(v^{-\ell}-\overline{v^{-\ell}})^{\nell}&=(v^{-\ell}-\overline{v^{-\ell}})^{\ell}-\overline{(v^{-\ell}-\overline{v^{-\ell}})^{\ell}}\\
	&=v-\overline{v^{-\ell}}-\overline{(v-\overline{v^{-\ell}})}\\
	&=v-\overline{v^{-\ell}}-(\overline{v}-\overline{v^{-\ell}})=v
\end{aligned}
\end{equation}
where we employed $(v^{-\ell})^\ell=v$ and $\overline{v}=0$. Similarly, one can verify $(\eta^{\nell})^{-\nell}=\eta$ for $\eta\in L^2(\cM_h)$ with $\overline{\eta}=0$. Besides, corresponding to  \eqref{error-nlift} there holds 
\begin{equation}\label{error-inv-nlift}
	\left|\overline{v^{-\ell}}\right|\le ch^2\|v^{-\ell}\|_h\sim ch^2\|v\|.
\end{equation}

\begin{lemma}\label{diff-inner}
	For any $w_h,v_h\in V_h$ with corresponding lifts $w_h^\ell$ and $v_h^\ell$, the following bounds hold:
	\begin{equation*}
		\left|(w_h,v_h)_{\cM_h}-(w_h^\ell,v_h^\ell)_{\cM}\right|\le ch^2\|w_h^\ell\| \|v_h^\ell\|,
	\end{equation*}
and
	\begin{equation*}
	\left|(\ngh w_h,\ngh v_h)_{\cM_h}-(\ng w_h^\ell,\ng v_h^\ell)_{\cM}\right|\le ch^2\|\ng w_h^\ell\| \|\ng v_h^\ell\|.
\end{equation*}
In particular, for $w_h,v_h\in V_h$ with vanishing mean values on $\cM_h$, the inequalities above still hold for $w_h^{\nell}$ and $v_h^{\nell}$ instead of $w_h^{\ell}$ and $v_h^{\ell}$.
\end{lemma}
\begin{proof}
	See Lemma 4.7 in \cite{dziuk2013finite}. The case of $w_h^{\nell}$ and $v_h^{\nell}$ follows from the fact $w_h^\ell=w_h^{\nell}+\overline{w_h^\ell}$ and $v_h^\ell=v_h^{\nell}+\overline{v_h^\ell}$ with $\overline{w_h^\ell},\overline{v_h^\ell}$ the corresponding mean values of $w_h^\ell,v_h^\ell$ on $\cM$, which can be  bounded by \eqref{error-nlift}.
\end{proof}

Let $\pi_h$ denote the $L^2$-orthogonal projection which maps $\forall \eta\in L^2(\cM_h)$ onto $\cP(\cM_h)\cap H^1(\cM_h)$, given by 
\begin{equation}\label{proj}
	(\eta-\pi_h \eta, v_h)_{\cM_h}=0, \quad \mbox{for }\forall v_h\in \cP(\cM_h)\cap H^1(\cM_h).
\end{equation}
For the $L^2$-projection we have the following stability and error estimate, as the Euclidean case.
\begin{lemma}[Proposition 4.4 in \cite{bonito2021approximation}]\label{bonito}
	The $L^2$-projection is both $L^2(\cM_h)$ and $H^1(\cM_h)$ stable, and in particular for $\eta\in H^r(\cM_h)$ with $r\in [0,1]$, there holds
	\begin{equation}\label{H-stability}
		\|\pi_h \eta\|_{H^r(\cM_h)}\le c \|\eta\|_{H^r(\cM_h)}.
	\end{equation}
Consequently for $r,\gamma\in [0,1]$ and $v\in H^{r+\gamma}(\cM)$ we have 
	\begin{equation}\label{L2-est-1}
		\|v^{-\ell}-\pi_h v^{-\ell}\|_{H^r(\cM_h)}\le c h^\gamma\|v\|_{H^{r+\gamma}(\cM)}.
	\end{equation}
\end{lemma}
The projection error \eqref{L2-est-1} also holds for $v^{-\nell}$(in the case of $\int_{\cM} vd\cM=0$) since  
\begin{equation}\label{nell-ell}
	v^{-\nell}-\pi_h v^{-\nell}=v^{-\ell}-\pi_h v^{-\ell}.
\end{equation}
Taking $v_h=1$ in \eqref{proj} we obtain that $\int_{\cM_h}\eta d\cM_h=0$ implies $\int_{\cM_h}{\pi_h \eta}d\cM_h=0$. 

To unify the analysis, we define
\begin{equation}\label{lift-new}
	\begin{aligned}
		\eta^{\tell},v^{-\tell}=\left\{
		\begin{matrix}
			\eta^{\ell}, v^{-\ell}  & \mbox{if } \Gamma\neq \emptyset, \mbox{ or }\Gamma=\emptyset, b(\vec{x})\not\equiv 0 \\
			\eta^{\nell},v^{-\nell}& \mbox{if } \Gamma= \emptyset 	\mbox{ and }b(\vec{x})\equiv 0
		\end{matrix}
		\right. .
	\end{aligned}
\end{equation}
For $\forall v\in H^r(\cM)$, taking $\eta=v^{-\tell}$ in \eqref{H-stability} and utilizing \eqref{equi-norm} or \eqref{equi-norm-new} it follows  that for $r\in [0,1]$
\begin{equation}\label{H-stability-1}
	\|(\pi_h v^{-\tell})^\tell\|_{H^r(\cM)}\le c \|v\|_{H^r(\cM)}, \quad v \in H^r(\cM)
\end{equation}
where we implicitly utilized $(v^{-\tell})^\tell=v$. Applying \eqref{equi-norm} or \eqref{equi-norm-new} for the projection error  we can obtain for $r,\gamma\in [0,1]$
\begin{equation}\label{L2-est-2}
	\|v-(\pi_h v^{-\tell})^\tell\|_{H^r(\cM)}\le c h^\gamma\|v\|_{H^{r+\gamma}(\cM)}, \quad \forall v\in H^{r+\gamma}(\cM).
\end{equation}
%%%%%%

Now we turn to the main error analysis. Appealing to Balakrishnan formula we know
\begin{equation}\label{Bala}
	u=\frac{\sin(\pi\alpha)}{\pi}\int_{0}^{\infty}\mu^{-\alpha}(\mu\cI +\cLg)^{-1}f\, d\mu,
\end{equation}
and 
\begin{equation}\label{Bala-h}
	u_h=\frac{\sin(\pi\alpha)}{\pi}\int_{0}^{\infty}\mu^{-\alpha}(\mu\cI +\cLg)^{-1}f_h\, d\mu.
\end{equation}
So to compare the error between $u$ and $u_h$, we introduce the auxiliary problem: for $f\in \dot{\mH}^\delta(\cM)$, find $w_\mu\in \mH^{1}(\cM)$ such that for $\forall v\in \mH^1(\cM)$
\begin{equation}\label{eqn-wmu}
	\langle \cLg w_\mu+ \mu w_\mu,v\rangle_\cM=\langle  f,v\rangle_{\cM}:=F(v), \quad \mbox{with }  \mu\ge 0.
\end{equation}
Hereafter we use $a_{\mu}(\cdot,\cdot)$ to represent the bilinear form corresponding to \eqref{eqn-wmu}. Note that $\mH^\sigma(\cM)$ coincides with $\dot{\mH}^\sigma(\cM)$ with equivalent norms for $\sigma\in[0,2]$, one can obtain
\begin{equation*}
	\|w_\mu\|_{H^{\sigma}(\cM)} \sim \|w_\mu\|_{\dot{H}^{\sigma}(\cM)}= \left\|\cLg^{\sigma/2} (\cLg +\mu\cI)^{-1}f\right\|\quad \mbox{for }\sigma\in[0,2].
\end{equation*}
Notice that for any scalar $\lambda\ge \lambda_1>0$ and $\mu\neq 0$, 
\begin{equation*}
	\frac{\lambda^{s}}{\lambda+\mu}=\mu^{s-1}\frac{(\lambda/\mu)^s}{\lambda/\mu +1}\le \min\{\mu^{s-1},\lambda_1^{s-1}\}\le c(\mu+1)^{s-1}\quad \mbox{for } \forall s\in[0,1].
\end{equation*} 
So expanding  $\cLg^{\sigma/2} (\cLg +\mu\cI)^{-1}f$ by mutually orthogonal eigenfunctions it follows for $f\in \dot{\mH}^{\delta}(\cM)$ and $\sigma\in [0,2]$ 
\begin{equation}\label{H-delta}
\|w_\mu\|_{H^{\sigma}(\cM)} \sim\|w_\mu\|_{\dot{H}^{\sigma}(\cM)}\le  c \left\{
\begin{matrix}
	&{(\mu+1)^{(\sigma-\delta)/2-1}} \|f\|_{\dot{H}^\delta(\cM)}, & \quad \sigma\ge \delta,\\
	&{(\mu+1)^{-1}} \|f\|_{\dot{H}^\delta(\cM)},& \quad \sigma< \delta.
\end{matrix}
\right. 
\end{equation}
Correspondingly we denote $w_{h,\mu}\in V_h$ as the solution of 
\begin{equation}\label{eqn-whmu}
	\cLgh w_{h,\mu} +\mu w_{h,\mu} =f_h, 
\end{equation}
with $f_h\in V_h$ an appropriate approximation of $f$ we shall clarify later.
And we let $a_{h,\mu}(\cdot,\cdot)$ denote the bilinear form corresponding to \eqref{eqn-whmu}. Testing \eqref{eqn-whmu} by $w_{h,\mu}$ one can get
\begin{equation}\label{Hh-s}
	\|w_{h,\mu}\|_h\le c (\mu+1)^{-1}\|f_h\|_h, \quad \mbox{and }\|w_{h,\mu}\|_{H^1(\cM_h)}\le c (\mu+1)^{-1/2}\|f_h\|_h.
\end{equation}

\subsection{Estimate for $w_\mu-w^{\tell}_{h,\mu}$ in $H^1(\cM)$}
Let $\|\cdot\|_\mu^2:=a_\mu(\cdot,\cdot)$ and $\|\cdot\|_{h,\mu}^2:=a_{h,\mu}(\cdot,\cdot)$. Consider the difference between $a_\mu(\cdot,\cdot)$ and $a_{h,\mu}(\cdot,\cdot)$, then the following lemma holds. 

\begin{lemma}\label{est-a-ah}
	Suppose $\cLgh$ is given by \eqref{paraFEM-2} and $a(\vec{x}),b(\vec{x})\in C^{\kappa}(\cM)$, then for $\forall w_h,v_h\in V_h$ and $\kappa\in[0,2]$ it holds that
	\begin{equation}\label{a-ah-r}
		\begin{aligned}
			\left|a_\mu(w_{h}^{\tell},v_h^{\tell})-a_{h,\mu}(w_{h},v_h)\right|&\le ch^\kappa\left( \|\ng  w_h^{\tell}\|+(\mu+1)^{1/2}\|w_h^{\tell}\| \right)\|v_h^{\tell}\|_\mu.
		\end{aligned}
	\end{equation}
\end{lemma}
\begin{proof}
	Let $\eta(\vec{x})$ represent $a(\vec{x})$ or $b(\vec{x})$, then it follows
	\begin{equation}\label{inter-eta}
		\|\eta-(I_h\eta^{-\ell})^\ell\|_{L^\infty(\cM)}\le c h^\kappa\|\eta\|_{C^{\kappa}(\cM)}.
	\end{equation}
	In fact, note that $I_h$ is piecewise linear, then it holds the stability $\|I_h\|_{C^{0}(K)\hookrightarrow L^\infty(K)}\le 1$ and the error estimate $\|\cI-I_h\|_{C^{2}(K)\hookrightarrow L^\infty(K)}\le ch_K^2$ on each $K\in\cK$. Denote $\tilde{K}=\{\theta(\vec{x})| \vec{x}\in K\}$, then by the chain rule(see \cite{dziuk1988finite} for the details) one can easily get $\|(\cdot)^{-\ell}\|_{C^2(K)}\le c \|(\cdot)\|_{C^2(\tilde{K})}$. Thus, lifting the stability result and the error estimate onto $\tilde{K}$ and taking the supremum over all $\tilde{K}$ then we can get \eqref{inter-eta} by interpolation. 

By transplanting the integral domain from $\cM_h$ to $\cM$ and repeating the proof of  \cite[Lemma 4.7]{dziuk2013finite} one can obtain the following estimates
\begin{equation}\label{a-2}
	\begin{aligned}
		&\left|\int_{\cM}a_h^{\ell}(\vec{x})\ng w_h^{\tell}(\ng v_h^{\tell})^Td\cM-\int_{\cM_h}a_h(\vec{x})\ngh w_h(\ngh v_h)^T d\cM_h\right|\\
		\le& c\left(\max_{\vec{x}\in \cM} a_h^\ell\right)^{1/2} h^2 \|\ng w_h^{\tell}\| \|(a^\ell_h)^{1/2}\ng v_h^{\tell}\|
	\end{aligned}
\end{equation}
and
\begin{equation}\label{a-3}
	\begin{aligned}
		&\left|\int_{\cM}(b_h^\ell+\mu) w_h^\tell v_h^{\tell} d\cM-\int_{\cM_h}(b_h+\mu) w_h v_h d\cM_h\right|\\
		\le &c\left(\mu+\max_{\vec{x}\in \cM}b_h^\ell\right)^{1/2}\, h^2 \|w_h^{\tell}\| \| (b_h^\ell+\mu)^{1/2}v_h^{\tell}\|.
	\end{aligned}
\end{equation}
Split $a(\vec{x})$ and $b(\vec{x})$ in $a_\mu(w_h^\tell,v_h^\tell)$ into $a(\vec{x})=(a(\vec{x})-a_h^\ell)+a_h^\ell$ and $b(\vec{x})=(b(\vec{x})-b_h^\ell)+b_h^\ell$, respectively. Then the desired estimate follows from triangle inequality, \eqref{inter-eta}--\eqref{a-3} and the fact that $\max_{\vec{x}\in\cM} a_h^\ell\le \bar{a}$, $\max_{\vec{x}\in\cM} b_h^\ell\le \bar{b}.$
\end{proof}
	\begin{remark}
		Replace $a_h^\ell(\vec{x})$, $a_h(\vec{x})$ in \eqref{a-2} and $b_h^\ell(\vec{x})$, $b_h(\vec{x})$ in \eqref{a-3} by $a(\vec{x})$, $a^{-\ell}(\vec{x})$ and $b(\vec{x})$, $b^{-\ell}(\vec{x})$, respectively, then one can directly get the estimate for $\cLgh$ given by \eqref{paraFEM-1}, where the only difference is that the term $h^\kappa$ is replaced by $h^2$. So we shall adopt \eqref{a-ah-r} to give  further analysis and one can get the error of \eqref{paraFEM-1} by just replacing $\kappa$ by $2$.
	\end{remark}

 Set $\e_{h,\mu}=w_{h,\mu}-z_h$ with $z_h=\pi_h (w_\mu^{-\tell})$, then 
\begin{equation}\label{eq-ehmu}
	\begin{aligned}
		\|\e_{h,\mu}\|_{h,\mu}^2&=a_{h,\mu}(w_{h,\mu},\e_{h,\mu})-a_{h,\mu}(z_h,\e_{h,\mu})\\
		&=F_h(\e_{h,\mu})-F(\e^{\tell}_{h,\mu})+F(\e^{\tell}_{h,\mu})-a_{h,\mu}(z_h,\e_{h,\mu})\\
		&=[F_h(\e_{h,\mu})-F(\e^{\tell}_{h,\mu})]+a_{\mu}(w_\mu-z_h^{\tell},\e^{\tell}_{h,\mu}) +[a_{\mu}(z_h^{\tell},\e^{\tell}_{h,\mu})-a_{h,\mu}(z_h,\e_{h,\mu})].
	\end{aligned}
\end{equation}
Recalling the definition of $a_\mu(\cdot,\cdot)$,  by Cauchy-Schwarz inequality we obtain
\begin{equation*}
	\begin{aligned}
		|a_{\mu}(w_\mu-z_h^{\tell},\e^{\tell}_{h,\mu})|&\le \left|\left(a(\vec{x})\ng (w_\mu-z_h^{\tell}), \ng \e^{\tell}_{h,\mu}\right)_{\cM}\right|+\left|\left((b+\mu)(w_\mu-z_h^{\tell}), \e^{\tell}_{h,\mu}\right)_\cM\right|\\
		&\le c \left(\|\ng (w_\mu-z_h^{\tell})\|+(\bar{b}+\mu)^{1/2}\|w_\mu-z_h^{\tell}\|\right)\|\e^{\tell}_{h,\mu}\|_\mu.
	\end{aligned}
\end{equation*}
Appealing to \eqref{L2-est-2} we get for $\gamma\in [0,1]$ and $\gamma'\in [0,2]$
\begin{equation}\label{term-2}
	\begin{aligned}
		|a_{\mu}(w_\mu-z_h^{\tell},\e^{\tell}_{h,\mu})|
		&\le c \left(h^\gamma\|w_\mu\|_{H^{1+\gamma}(\cM)}+(1+\mu)^{1/2}h^{\gamma'}\|w_\mu\|_{H^{\gamma'}(\cM)} \right)\,\|\e^{\tell}_{h,\mu}\|_\mu.
	\end{aligned}
\end{equation}
Taking  $w_h=z_h$ and $v_h =\e_{h,\mu}$ in \eqref{a-ah-r}, inserting the obtained inequality and \eqref{term-2} into \eqref{eq-ehmu}, we get for $\gamma\in [0,1]$ and $\gamma'\in[0,2]$
\begin{equation}\label{est-ehmul}
	\begin{aligned}
		\|\e_{h,\mu}^{\tell}\|_{\mu}\le& c\left(h^\gamma\|w_\mu\|_{H^{1+\gamma}(\cM)}+(1+\mu)^{1/2}h^{\gamma'}\|w_\mu\|_{H^{\gamma'}(\cM)} \right)\\
		&+ch^\kappa\left(\|w_\mu\|_{H^{1}(\cM)}+(1+\mu)^{1/2}\|w_\mu\|_{L^2(\cM)} \right) +r(f,f_h)
	\end{aligned}
\end{equation}
where 
\begin{equation*}
	r(f,f_h)=\frac{\left|F_h(\e_{h,\mu})-F(\e^{\tell}_{h,\mu})\right|}{\|\e_{h,\mu}^{\tell}\|_{\mu}}.
\end{equation*}
One can choose $f_h\in V_h$ such that 
\begin{equation}\label{f_h}
	\int_{\cM_h} f_h v_h d\cM_h=\int_{\cM} f v_h^{\tell} d\cM, \quad \mbox{for }\forall v_h\in V_h.
\end{equation}
This is the way choosing $f_h$ in \cite{bonito2021approximation} which vanishes $r(f,f_h)$. Here we make the following assumptions on $f_h$:
\begin{assumption}\label{assump-f}
	For given $f\in \dot{\mH}^\delta(\cM)$, we assume that $f_h\in V_h$ satisfies the stability condition $\|f_h\|_h\le c \|f\|_{\dot{H}^\delta(\cM)}$ and processes the approximation property
		\begin{equation}\label{appro-f}
			\|f_h^{\tell}-f\|_{H^{-1}(\cM)}\le ch^{1+\delta}\|f\|_{\dot{H}^{\delta}(\cM)},\quad 0\le \delta\le 1.
		\end{equation}
\end{assumption}
\begin{remark}
	Obviously, $f_h=\pi_hf^{-\tell}$ satisfies the assumption above. In fact, Assumption \ref{assump-f} is sufficient for the later analysis but it is not necessary. The analysis hereafter will show that the approximation property of $f_h$ has effect only on  $\bar{R}_4$ in \eqref{est-L2-emu}. In some special situation, for instance if $f\in \mH^2(\cM)$, then 
	\begin{equation*}
		\|f_h^{\tell}-f\| \le ch^{2}\|f\|_{\dot{H}^{2}(\cM)}
	\end{equation*}
is enough to derive the desired estimates as well. That is, one can obtain $f_h$ by interpolation. For the interpolation error one may refer to \cite[Lemma 4.3]{dziuk2013finite}.
\end{remark}
Using triangle inequality, H\"older inequality and \eqref{appro-f} we obtain
\begin{equation}\label{est-F}
\begin{aligned}
		\left|F_h(\e_{h,\mu})-F(\e^{\tell}_{h,\mu})\right|&\le\left|(f_h,\e_{h,\mu})_{\cM_h}-(f_h^{\tell},\e^{\tell}_{h,\mu})_{\cM}\right|+\left|(f_h^{\tell},\e^{\tell}_{h,\mu})_{\cM}-(f,\e_{h,\mu}^{\tell})_{\cM}\right|\\
		&\le  c \left(h^2 \|f_h^{\tell}\|_{H^{-1}(\cM)}+h^{1+\delta}\|f\|_{\dot{H}^{\delta}(\cM)}\right) \|\e_{h,\mu}^{\tell}\|_{H^1(\cM)}
\end{aligned}
\end{equation}
where the first term of the last line follows from
\begin{equation*}
	(f_h,\e_{h,\mu})_{\cM_h}-(f_h^{\tell},\e^{\tell}_{h,\mu})_{\cM}=\int_{\cM}f_h^{\tell} e^{\tell}_{h,\mu}(\sqrt{|g|^{-1}}-1) d\cM.
\end{equation*}
Again recalling that  $\|1-\sqrt{|g|}\|_{L^\infty(\cM_h)}\le ch^2$ we obtain
\begin{equation}\label{est-rf}
	r(f,f_h)=\frac{\left|F_h(\e_{h,\mu})-F(\e^{\tell}_{h,\mu})\right|}{\|\e_{h,\mu}^{\tell}\|_{\mu}}\le c(\mu+1)^{-1/2}h^{1+\delta}\|f\|_{{\dot{H}}^{\delta}(\cM)}
\end{equation}
where we used $\|f_h^{\tell}\|_{H^{-1}(\cM)}\le c\|f\|$ which can be obtained from  \eqref{appro-f}. Apply  triangle inequality 
$$\|w_\mu -w_{h,\mu}^{\tell}\|_{\mu}:=
\|e_\mu\|_\mu  \le \|w_\mu-z_h^{\tell}\|_\mu +\|\bar{e}_{h,\mu}^{\tell}\|_\mu.$$
The first part on the right hand side is the projection error which can be bounded by  \eqref{L2-est-2}. So appealing to \eqref{est-ehmul} and \eqref{est-rf}, and utilizing the regularity estimate \eqref{H-delta} we get for $f\in \dot{\mH}^\delta(\cM)$, $\gamma\in [0,1]$ with $\gamma+1\ge \delta$ and $\gamma'\in[\delta,2]$
\begin{equation}\label{est-ehmu-1}
\begin{aligned}
	\|e_\mu\|_{\mu}\le& c h^\gamma (\mu+1)^{(\gamma-\delta-1)/2}\|f\|_{\dot{H}^{\delta}(\cM)}+c h^{\gamma'} (\mu+1)^{(\gamma'-\delta-1)/2}\|f\|_{\dot{H}^{\delta}(\cM)}\\
	& +ch^{\kappa_\delta}(\mu+1)^{-1/2}\|f\|_{\dot{H}^{\delta}(\cM)},
\end{aligned}
\end{equation}
where $\kappa_\delta=\min\{\kappa,1+\delta,2\}$.
\subsection{Estimate for $w_\mu-w^{\tell}_{h,\mu}$ in $L^2(\cM)$}
 We shall use the Aubin-Nitsche trick to derive the $L^2(\cM)$-error. Towards this end,  consider finding $y_\mu\in \mH^1(\cM)$ such that for $\forall v\in \mH^1(\cM)$ it holds
\begin{equation} \label{eqn-dual}
	a_\mu(y_\mu,v)=( e_\mu,v)_\cM,\quad\mbox{with }e_\mu=w_\mu-w^{\tell}_{h,\mu},\quad  \mu\ge0.
\end{equation}
Similar to \eqref{H-delta} we can get for $\sigma\in [0,2]$
\begin{equation}\label{H-delta-y}
	\|y_\mu\|_{H^{\sigma}(\cM)}\le c (\mu+1)^{\sigma/2-1}\|e_\mu\|.
\end{equation}
Denote $y_{h,u}=\pi_h( y_\mu^{-\tell})$ and recall $z_h=\pi_h (w_{\mu}^{-\tell})$. Test \eqref{eqn-dual} by $e_\mu$  we obtain
\begin{equation}\label{L2-ehmu}
	\begin{aligned}
		\|e_{\mu}\|^2=&a_\mu(e_\mu,y_\mu)=a_\mu(e_\mu,y_\mu-y_{h,\mu}^\tell)+a_\mu(e_\mu,y_{h,\mu}^\tell)\\
		=&a_\mu(e_\mu,y_\mu-y_{h,\mu}^\tell) +a_\mu( w_\mu-w_{h,\mu}^\tell,y_{h,\mu}^\tell)\\
		=&a_\mu(e_\mu,y_\mu-y_{h,\mu}^\tell) +\left[F(y_{h,\mu}^\tell)-F_h(y_{h,\mu})\right]\\
		&+\left[a_{h,\mu}(w_{h,\mu},y_{h,\mu})-a_\mu(w_{h,\mu}^\tell,y_{h,\mu}^\tell)\right].
	\end{aligned}
\end{equation}
For the first term on the right hand side of \eqref{L2-ehmu}, appealing to \eqref{L2-est-1}  we obtain for $\tgamma\in[0,1]$
\begin{equation*}
\begin{aligned}
	|a_\mu(e_\mu,y_\mu-y_{h,\mu}^\tell)|&\le \|e_\mu\|_\mu \|y_{\mu}-y_{h,\mu}^\tell\|_{\mu} \\
	&\le c \|e_\mu\|_{\mu} \left(h^\tgamma\|y_\mu\|_{H^{1+\tgamma}(\cM)}+(\mu+1)^{1/2}h^\tgamma\| y_\mu\|_{H^\tgamma(\cM)}\right).
\end{aligned}
\end{equation*}
 Appealing to the regularity of $y_\mu$, we immediately obtain $\tgamma\in[0,1]$
\begin{equation}\label{RHS-1}
	|a_\mu(e_{\mu},y_\mu-y_{h,\mu}^\tell)|\le c h^\tgamma(\mu+1)^{(\tgamma-1)/2} \|e_\mu\|_\mu \|e_\mu\|.
\end{equation}
Recalling Lemma \ref{diff-inner}, \eqref{H-stability-1} together with the fact that $y_{h,\mu}=\pi_h y_\mu^{-\tell}$ it follows
\begin{equation}\label{RHS-2}
	\left|F(y_{h,\mu}^\tell)-F_h(y_{h,\mu})\right|\le c h^2\|y_{h,\mu}^\tell\| \|f\|\le c h^2\|y_{\mu}\| \|f\|\le c h^2(\mu+1)^{-1}\|f\| \|e_\mu\|.
\end{equation}
Appealing to \eqref{a-ah-r} and  \eqref{H-stability-1}, together with the regularity pickups for $y_{\mu}$ and $w_{h,\mu}$ we get
\begin{equation}\label{RHS-3}
	\begin{aligned}
		&\left|a_{h,\mu}(w_{h,\mu},y_{h,\mu})-a_\mu(w_{h,\mu}^\tell,y_{h,\mu}^\tell)\right|\le c h^{\min(\kappa,2)}(\mu+1)^{-1}\|f\|_{\dot{H}^\delta(\cM)} \|e_\mu\|
	\end{aligned}
\end{equation}
where we utilized the stability assumption on $f_h$ in Assumption \ref{assump-f}. Inserting \eqref{RHS-1}, \eqref{RHS-2} and \eqref{RHS-3} into \eqref{L2-ehmu} and appealing to \eqref{est-ehmu-1} we arrive at
	\begin{equation}\label{est-L2-emu}
	\begin{aligned}
		\|e_\mu\|\le& c h^{\min(\kappa,2)}(\mu+1)^{-1} \|f\|_{\dot{H}^\delta(\cM)} +ch^\tgamma (\mu+1)^{(\tgamma-1)/2}\|e_\mu\|_\mu \\
		\le&  ch^{\min(\kappa,2)}(\mu+1)^{-1} \|f\|_{\dot{H}^\delta(\cM)}+ ch^{\gamma+\tgamma} (\mu+1)^{(\gamma+\tgamma-\delta)/2-1} \|f\|_{\dot{H}^\delta(\cM)} \\
		&+ch^{\gamma'+\tgamma} (\mu+1)^{(\gamma'+\tgamma-\delta)/2-1} \|f\|_{\dot{H}^\delta(\cM)}
		+ ch^{\tgamma+\kappa_\delta}(\mu+1)^{\tgamma/2-1}  \|f\|_{\dot{H}^{\delta}(\cM)}\\
		:=&\bar{R}_1 +\bar{R}_2 +\bar{R}_3+\bar{R}_4
	\end{aligned}
	\end{equation}
where $\gamma,\tgamma\in [0,1]$ with $\gamma+1\ge \delta$ and $\gamma'\in[\delta,2]$.
The analysis above implies the following theorem.
\begin{theorem}\label{The-est-h}
	Suppose $a(\vec{x})\in H^1(\cM)\cap C^{\kappa}(\cM)$, $b(\vec{x})\in C^{\kappa}(\cM)$, $f\in \dot{\mH}^{\delta}(\cM)$ and $f_h$ satisfies Assumption \ref{assump-f}, where $\delta\in[0,2]$ and $\kappa=\min\{2\alpha+\delta,2\}$. Let $u=\cLg^{-\alpha}f$ and $u_h=\cLgh^{-\alpha}f_h$, then for  $h$ small enough
	\begin{equation*}
		\|u-u_h^{\ell}\|\le 
		c \left\{
		\begin{matrix}
			 |2\alpha+\delta-2|^{-1}h^{\min(\delta+2\alpha,2)}\|f\|_{\dot{H}^{\delta}(\cM)}& \quad \alpha+\delta/2\neq 1,\\
			 h^{2}|\ln h|\|f\|_{\dot{H}^{\delta}(\cM)}& \quad \alpha+\delta/2= 1,
		\end{matrix}
	\right.
	\end{equation*}
with $c$ independent of $\alpha$ and $h$. %Furthermore, the estimates also hold for $\|u-u_h^\ell\|$. 
%For the case of $\delta+2\alpha>2$,   the logarithmic term vanishes.
\end{theorem}
\begin{proof}
 	Recalling \eqref{Bala} and \eqref{Bala-h} we know
	\begin{equation*}
		\|u-u_h^\tell\|\le \frac{\sin(\pi\alpha)}{\pi}\int_{0}^{\infty} \mu^{-\alpha} \|e_\mu\| d\mu \le \frac{\sin(\pi\alpha)}{\pi}\sum_{j=1}^{4}\int_{0}^{\infty} \mu^{-\alpha}  \bar{R}_j d\mu:=\frac{\sin(\pi\alpha)}{\pi}\sum_{j=1}^{4}I_j.
	\end{equation*}
  For the $I_1$, one can easily derive the upper bound $c_{\alpha,1} h^{\min(\kappa,2)} \|f\|_{\dot{H}^\delta(\cM)}$. And for $I_j(j=2,3,4)$ we split it into 
	\begin{equation}\label{I_2}
		I_j=\int_{0}^{1}\mu^{-\alpha}\bar{R}_j d\mu +\int_{1}^{h^{-2}}\mu^{-\alpha}\bar{R}_j d\mu +\int_{h^{-2}}^{\infty}\mu^{-\alpha}\bar{R}_j d\mu:=\sum_{k=1}^{3}I_{j,k}.
	\end{equation}
	For $I_{j,1}$ we set $\gamma=\tgamma=1$, $\gamma'=2$ then the bound $c_{\alpha,2}h^{\min(1+\kappa_\delta,2)} \|f\|_{\dot{H}^\delta(\cM)}$ follows. For $I_{2,2}$, if $\alpha+\delta/2\neq 1$ we take  $\gamma = \tgamma=1$,  then 
	\begin{equation*}
		I_{2,2}\le \frac{c}{|2\alpha+\delta-2|} h^{\delta+2\alpha}\|f\|_{\dot{H}^{\delta}(\cM)},
	\end{equation*}
and if $\alpha+\delta/2=1$, then by taking $\gamma=\tgamma=1-\epsilon/2$ it follows
\begin{equation*}
		I_{2,2}\le \frac{c}{\epsilon} h^{2-\epsilon}\|f\|_{\dot{H}^{\delta}(\cM)}.
\end{equation*}
For $I_{2,3}$ and $I_{3,3}$, we take $\gamma=\max(\delta-1,0)$, $\tgamma=0$, $\gamma'=\delta$ then
\begin{equation*}
	I_{j,3}\le c_{\alpha,3} h^{\delta+2\alpha}\|f\|_{\dot{H}^{\delta}(\cM)},\quad j=2,3.
\end{equation*}
For $I_{3,2}$, by taking $\gamma'=2,\tgamma=1$ we obtain% for $\delta+2\alpha\le 2$
\begin{equation*}
	I_{3,2}\le \frac{c}{|3-2\alpha-\delta|} h^{\min(\delta+2\alpha,3)}\|f\|_{\dot{H}^{\delta}(\cM)},
\end{equation*}
and when $\delta+2\alpha>2$ we can use $\delta'=2-2\alpha$ instead of $\delta$ as the smoothness of $f$ to get the uniform bounded(to avoid $3-2\alpha-\delta\rightarrow 0$). So it follows
\begin{equation*}
	I_{3,2}\le c h^{\min(\delta+2\alpha,2)}\|f\|_{\dot{H}^{\delta}(\cM)},
\end{equation*}
Finally for $I_{4,2}$ and $I_{4,3}$, take  $\tgamma=\alpha$ then we arrive at 
	\begin{equation*}
		I_{4,2}+I_{4,3}\le c_{\alpha,4} h^{\alpha+\kappa_\delta}\|f\|_{\dot{H}^\delta(\cM)}.
	\end{equation*}
In fact, the constants $\{c_{\alpha,k}\},k=1,\cdots,4$ are dependent on $\alpha$ and may blow up for $\alpha\rightarrow 0^+$ and $1^-$. However, taking into account of the multiplier $\sin(\pi\alpha)$, trivial calculations show $c_{\alpha,k}\sin(\pi\alpha)\le c$ with $c$ independent of $\alpha$. Thus we end our proof by choose $\epsilon=|\ln(h)|^{-1}$ in the case of $\alpha+\delta/2=1$. 

Note that $u_h^\ell-u_h^\tell$ is either $0$ or a constant bounded by $ch^2\|u_h\|_h$(see \eqref{error-nlift}), so using the regularity of $u_h$ and Assumption \ref{assump-f}  we know $\|u_h\|_h\le \|f_h\|_{h}\le c\|f\|_{\dot{H}^\delta(\cM)}$. So the estimates for $\|u-u_h^{\ell}\|$ follow from triangle inequality $\|u-u_h^\ell\|\le \|u-u_h^\tell\|+\|u^\tell-u_h^\ell\|$. 
\end{proof}

Combining Theorem \ref{err-t} with Theorem \ref{The-est-h}, and applying triangle inequality together with $\|u_h -U_{L+1}\|_h \sim \|u_h^\ell -U^\ell_{L+1}\|,$ we arrive at the following Corollary:
\begin{corollary}
	Suppose the condition in Theorem \ref{The-est-h} is satisfied and the mesh $\{t_l\}_{l=0}^{L+1}$ is constructed by \eqref{mesh_t}. Denote by $N_s$ the number of total solves, then for $h$ small enough, the analytic solution $u$ and the numerical solution $U_{L+1}$ obtained from our scheme satisfy
	\begin{equation*}
		\|u-U_{L+1}^\ell\|\le c\left(\hat{c} \tlambda^{-\alpha}{32^{-\frac{N_s}{\lceil \log_2 (\lambda_{h,max}/\tlambda)\rceil}}}+c_hh^{\min(\delta+2\alpha,2)}\right) \|f\|_{\dot{H}^{\delta}(\cM)},
	\end{equation*}
	where $c$ is independent of $\alpha,h$,  and $\hat{c}$ is given in Corollary \ref{err-t-cor}, and $c_h=1$ if $\alpha+\delta/2\neq1$ or $c_h=|\ln h|$.
\end{corollary}

\section{Numerical examples}
In this section we present numerical examples to verify the theoretical analysis and to show the  efficiency and robustness of our method. All the programs are run in MATLAB, and the meshes in Example \ref{2d-sphere} and Example \ref{torus} are generated by Gmsh. 

Our algorithm needs the knowledge of $\lambda_{h,min}$ and $\lambda_{h,max}$. It is known that $\lambda_{h,min}\ge \lambda_{min}$ with $\lambda_{min}$ the minimal eigenvalue of $\cLg$, and $\lambda_{h,max} \le \Lambda =ch^{-2}$ with $c$ sufficiently large. In all our tests, instead of $\lambda_{h,max}$ we use an upper bound $\Lambda$. However, it is preferable to use a good 
practically feasible bound $\Lambda$, which can be obtained by a few iterations of power method.  As Theorem \ref{err-t} shows, the error is proportional to the factor 
$\tlambda^{-\alpha}$, so it is always better to choose  $\tlambda$ as close as possible to $\lambda_{h,min}$. Generally one can set $\tlambda=\lambda_{min}$.

To implement the formulas of Algorithm \ref{a-basic} one needs the roots $t_1(-\alpha, \alpha) < \dots < t_m(-\alpha, \alpha)$
and $t_1(\alpha, -\alpha) < \dots < t_m(\alpha, -\alpha)$ of the Jacobi polynomials $J_m^{(-\alpha, \alpha)}(t) $ and $J_m^{(\alpha, -\alpha)}(t)$,
respectively. For example, for $m=1$ it is well known that $t_1(\alpha, -\alpha)=\frac12(1-\alpha)$ and $t_1(-\alpha, \alpha)=\frac12(1+\alpha)$.
For various $m$ and $\alpha$ the roots could be obtained by  directly computing the eigenvalues of a tridiagonal matrix \cite[pp.84]{shen2011spectral} or by the fast algorithm \texttt{jacpts} in \texttt{chebfun} for very large $m$ (see \cite{hale2013fast} for details).

\begin{example}\label{2d-Rec}
	To verify the robustness of our algorithm we  take the checkerboard problem on   domain $[-1,1]\times[-1,1]$ with homogeneous boundary for the first test, say, $\cM$ is a plane square. We apply a standard finite difference scheme with more than $10^6$ degree of freedom to discrete  $\cLg$ with $a(\vec{x})=1,b(\vec{x})=0$  and  
	\begin{equation*}
		f=\left\{ \begin{matrix}
			1,& \mbox{for } \ \ x_1x_2\ge0;\\
%			0,& \mbox{for } \ \ x_1x_2=0;\\
			-1& \mbox{for } \ \ x_1x_2<0.\\
		\end{matrix}
		\right.
	\end{equation*}
\end{example}
The mesh we use is geometrical refined around the boundary and along $x_1,x_2=0$:  we first divide each direction into $N_1=2N_0=1000$ intervals $\{I_k\}_{k=1}^{N_1}$, then refine $I_k$ with  $k=1,N_0,N_0+1,N_1$ by adding $p=12$ nodes in each of them, denoting by $\{x_{k}^n\}_{n=1}^p$ which are exponentially clustered correspondingly at $-1^+,0^{-},0^{+}$ and $1^-$ with a speed of $2^{-n}$. One can obtain $\Lambda\approx 3.355\times10^{13}$ and the total degree of freedom is $1047^2=1096209> 10^6$ for such $\cLgh$. It is also worth to point out that in each direction the ratio of the mesh size is ${h_{max}}/{h_{min}}=2^p=4096$. The two subplots in Fig.\ref{fig-2dRec-error-delta} show the results under $\tlambda=0.01$ and $\tlambda=\lfloor\frac{\pi^2}{2}\rfloor=4$ respectively, and the number of steps we need in temporal direction correspondingly is $L+1=43$ and $L+1=52$. One can observe that the results consist with our theoretical analysis. Fig.\ref{fig-2d-FDM} presents the numerical solution for $\alpha=0.01,\alpha=0.5$ and $\alpha=0.99$, respectively.
\begin{figure}[!htp]
	\centering
	\begin{tabular}{c c}
		\subfigure[$\tlambda=0.01$]{\includegraphics[width=0.40\textwidth,trim=140 210 160 230, clip]{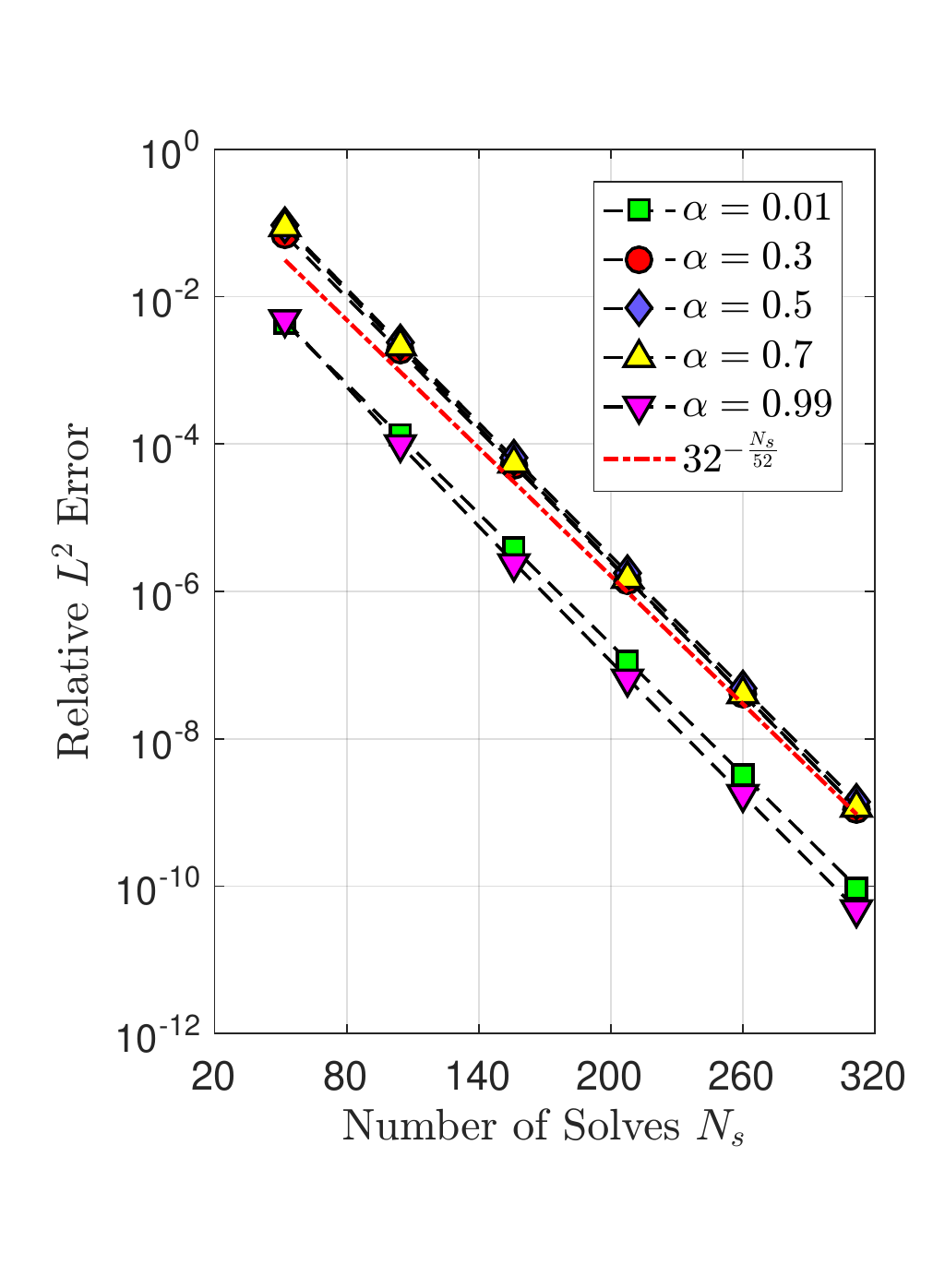}}&
		\subfigure[$\tlambda=4$]{\includegraphics[width=0.40\textwidth,trim=140 210 160 230, clip]{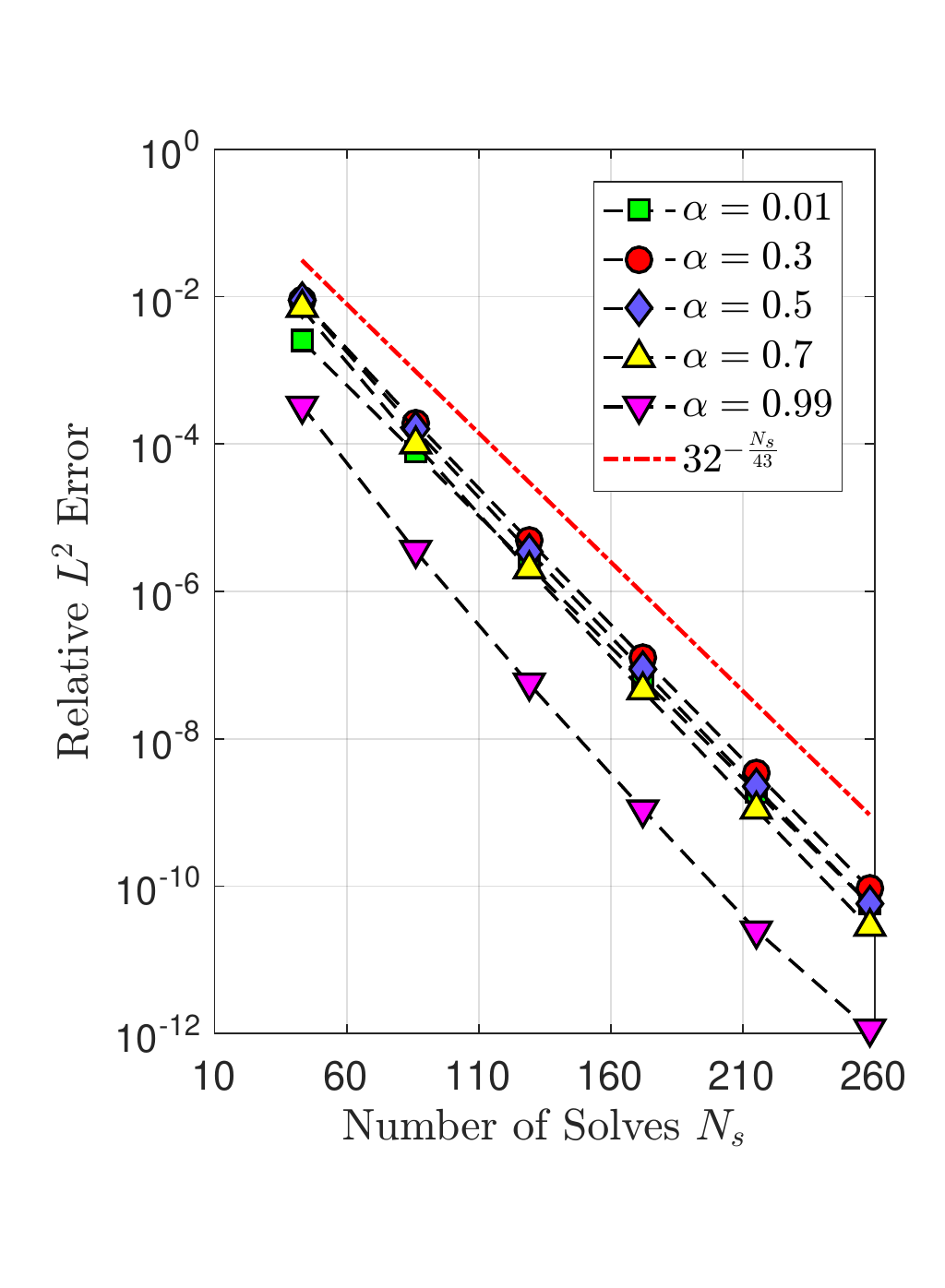}}
	\end{tabular}
	\caption{$\|u_h-U_{L+1}\|_h/\|u_h\|_h$ Example \ref{2d-Rec} for different $\tlambda$}
	\label{fig-2dRec-error-delta}
\end{figure}

\begin{figure}[!htp]
	\centering
	\begin{tabular}{c c c}
		\subfigure[$\alpha=0.01$]{\includegraphics[width=0.32\textwidth,trim=100 210 100 230, clip]{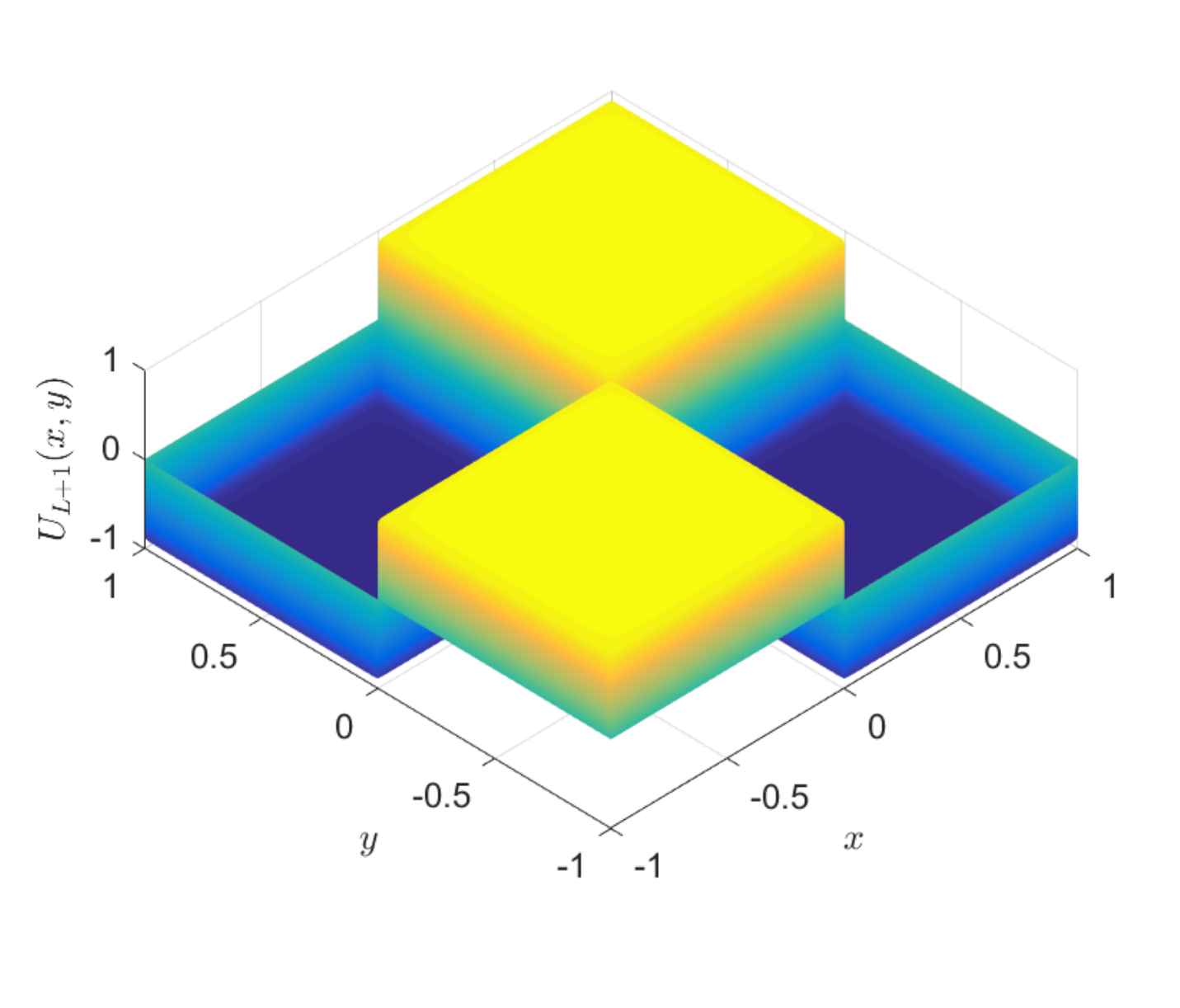}}&
		
		\subfigure[$\alpha=0.5$]{\includegraphics[width=0.32\textwidth,trim=100 210 100 230, clip]{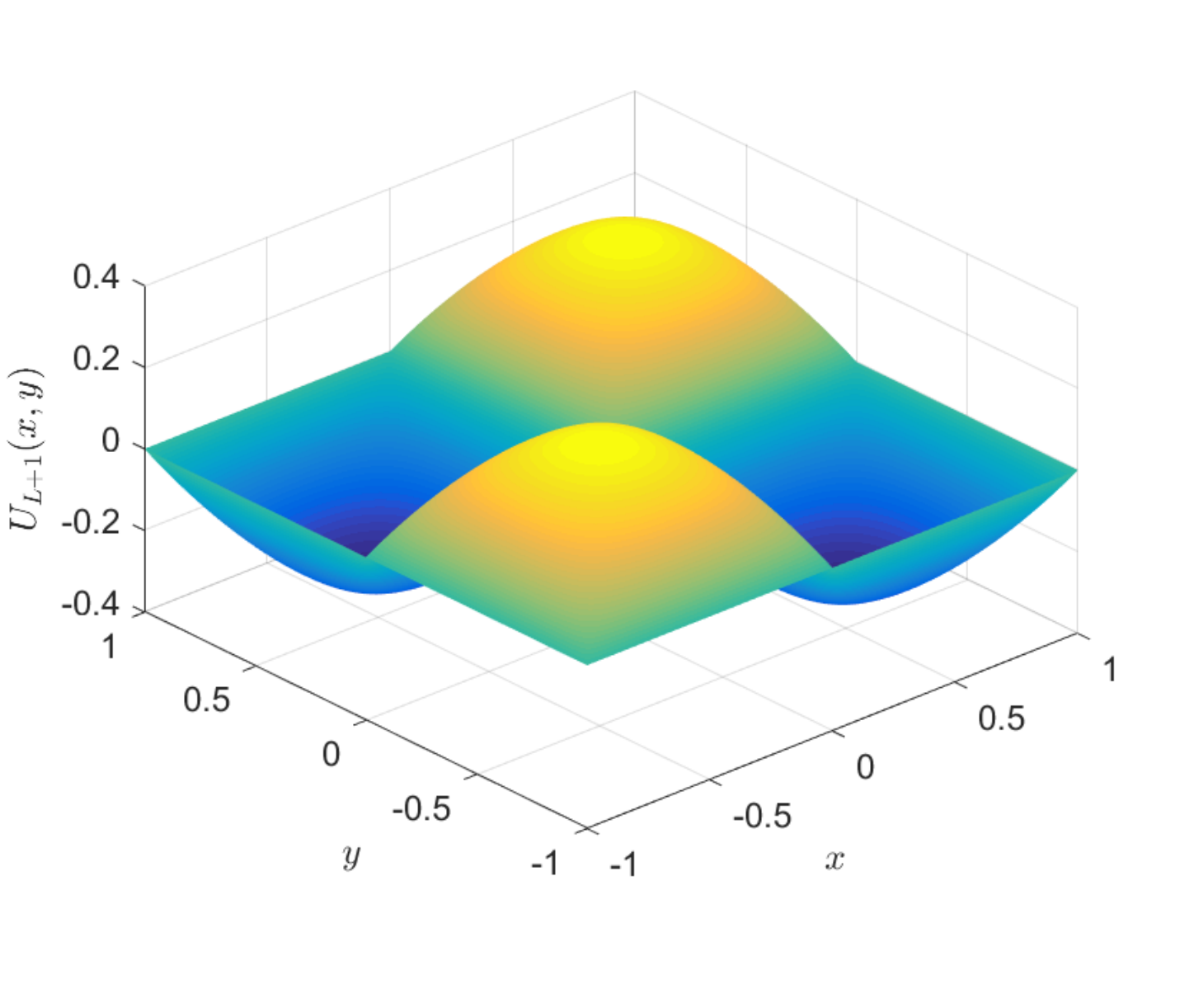}}&
		
		\subfigure[$\alpha=0.99$]{\includegraphics[width=0.32\textwidth,trim=100 210 100 230, clip]{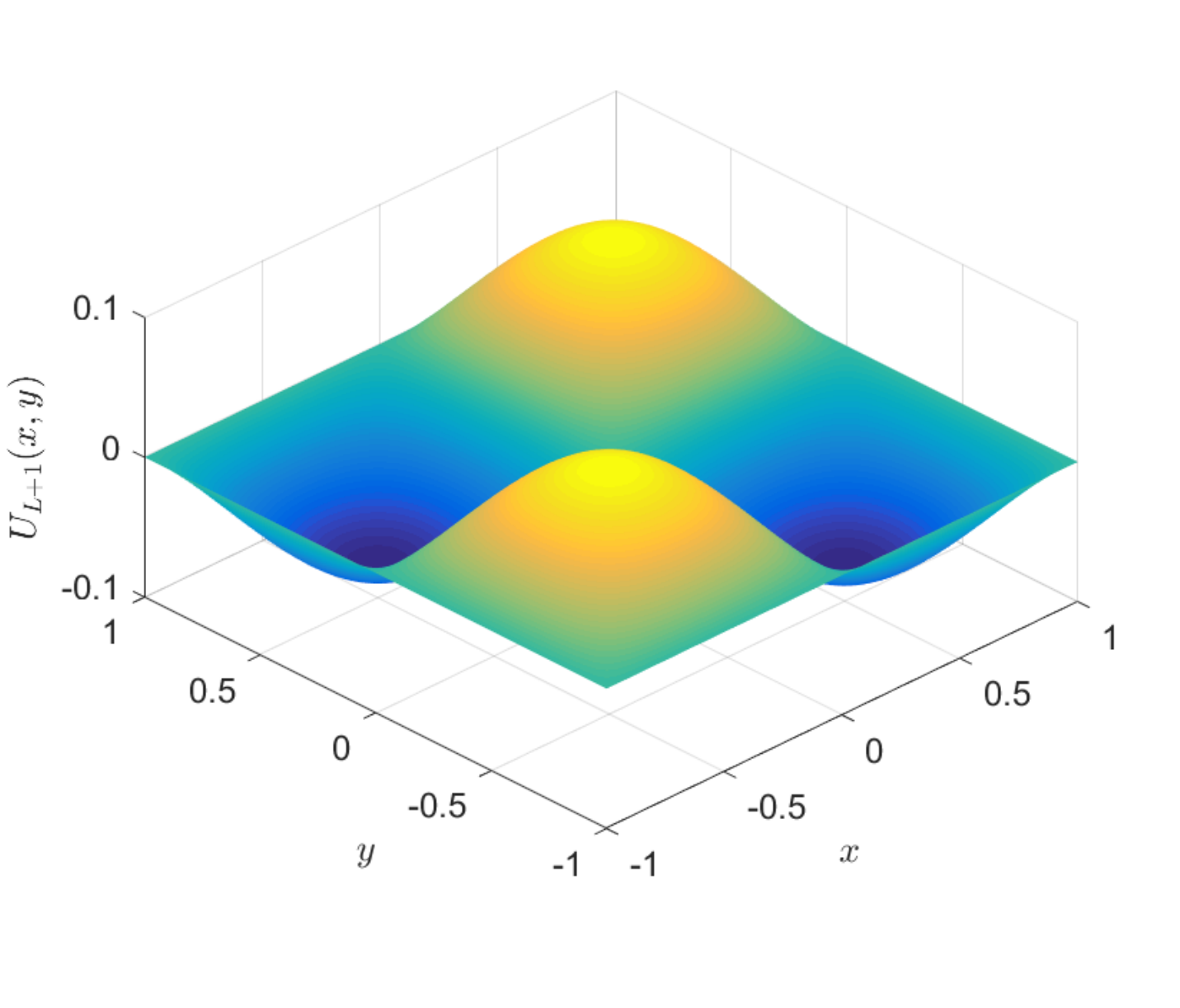}}
	\end{tabular}
	\caption{Numerical solutions for Example \ref{2d-Rec} under different $\alpha$}
	\label{fig-2d-FDM}
\end{figure}
We would also like to compare our scheme with Bonito-Pasciak scheme \cite{bonito2015numerical} and Aceto-Novati scheme \cite{aceto2019rational}. We employ Example \ref{2d-Rec} with $N_0=25,p=12$ to do the test.  Running few iterations of power method gives $\Lambda=7.5\times 10^{10}$. We still set $\hat{\lambda}=4$ in our scheme. To run the Aceto-Novati scheme, one needs to pick up an external parameter for each $\alpha$ and $N_s$, which actually requires an estimate of $\lambda_{h,min}$. Here we also take $4$ as its approximation. The  results are presented in Fig.\ref{fig-compare}. One can observe that our scheme gives better results for $N_s$ large enough. In fact, as a function of $N_s$, the error of Bonito-Pasciak scheme is $\mathcal{O}(e^{-\pi\sqrt{\alpha(1-\alpha)N_s}})$, which indicates that the scheme will degenerate for $\alpha$ close to $1$ or $0$. As for  Aceto-Novati scheme, the error bound given by the authors is
$$
C\sin(\alpha\pi)\lambda_{h,max}^{-\alpha/2}\exp\left({-4N_s\left(\frac{\lambda_{h,min}}{\lambda_{h,max}}\right)^{1/4}}\right),
$$
which implies slow convergence when  the condition number of $\cLgh$ is large, especially in the case of  $\alpha<0.5$.
\begin{figure}[!htp]
	\centering
	\begin{tabular}{c c c}
		\subfigure[$\alpha=0.1$]{\includegraphics[width=0.32\textwidth,trim=140 210 170 250, clip]{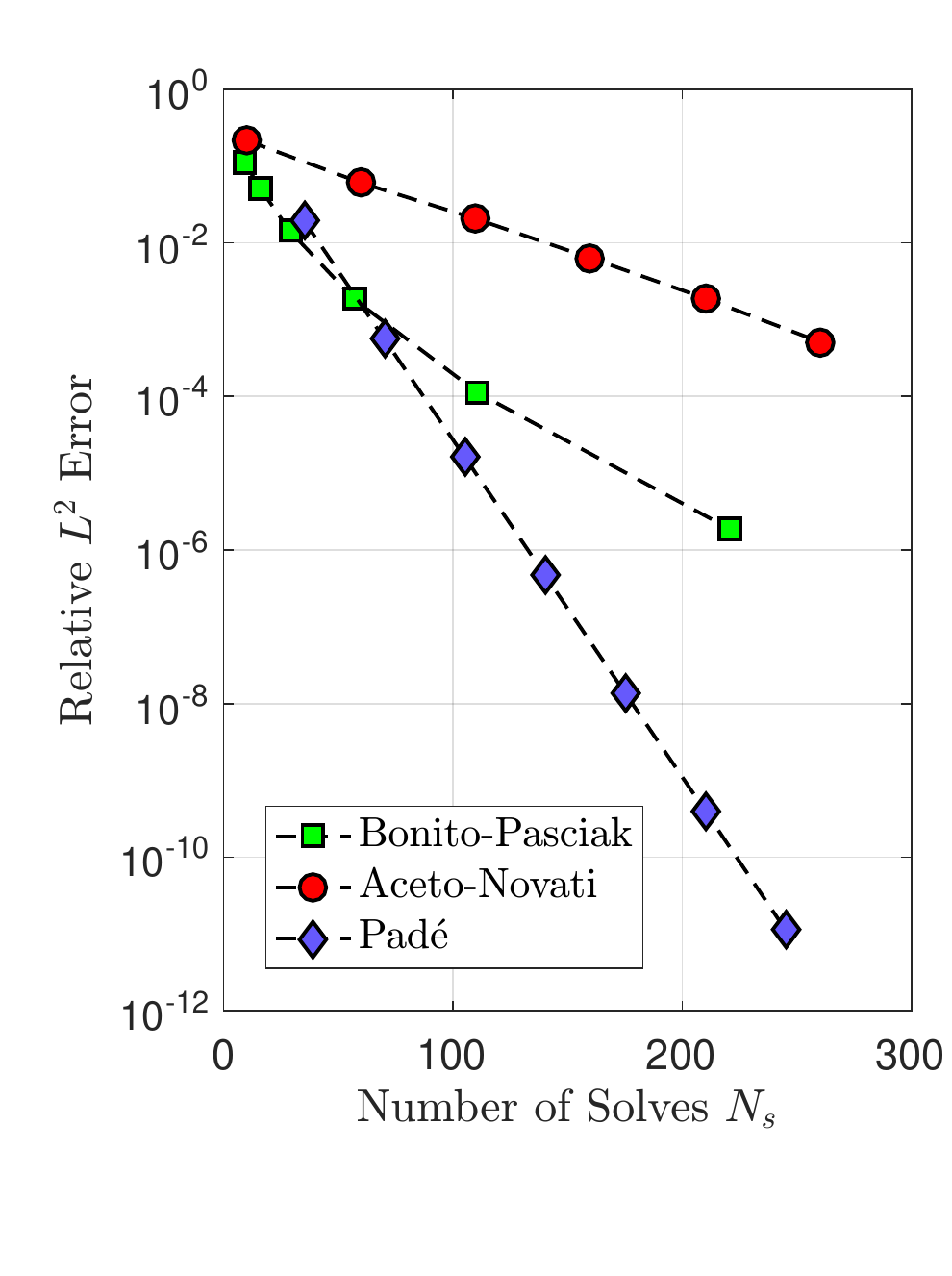}}&
		
		\subfigure[$\alpha=0.5$]{\includegraphics[width=0.32\textwidth,trim=140 210 170 250, clip]{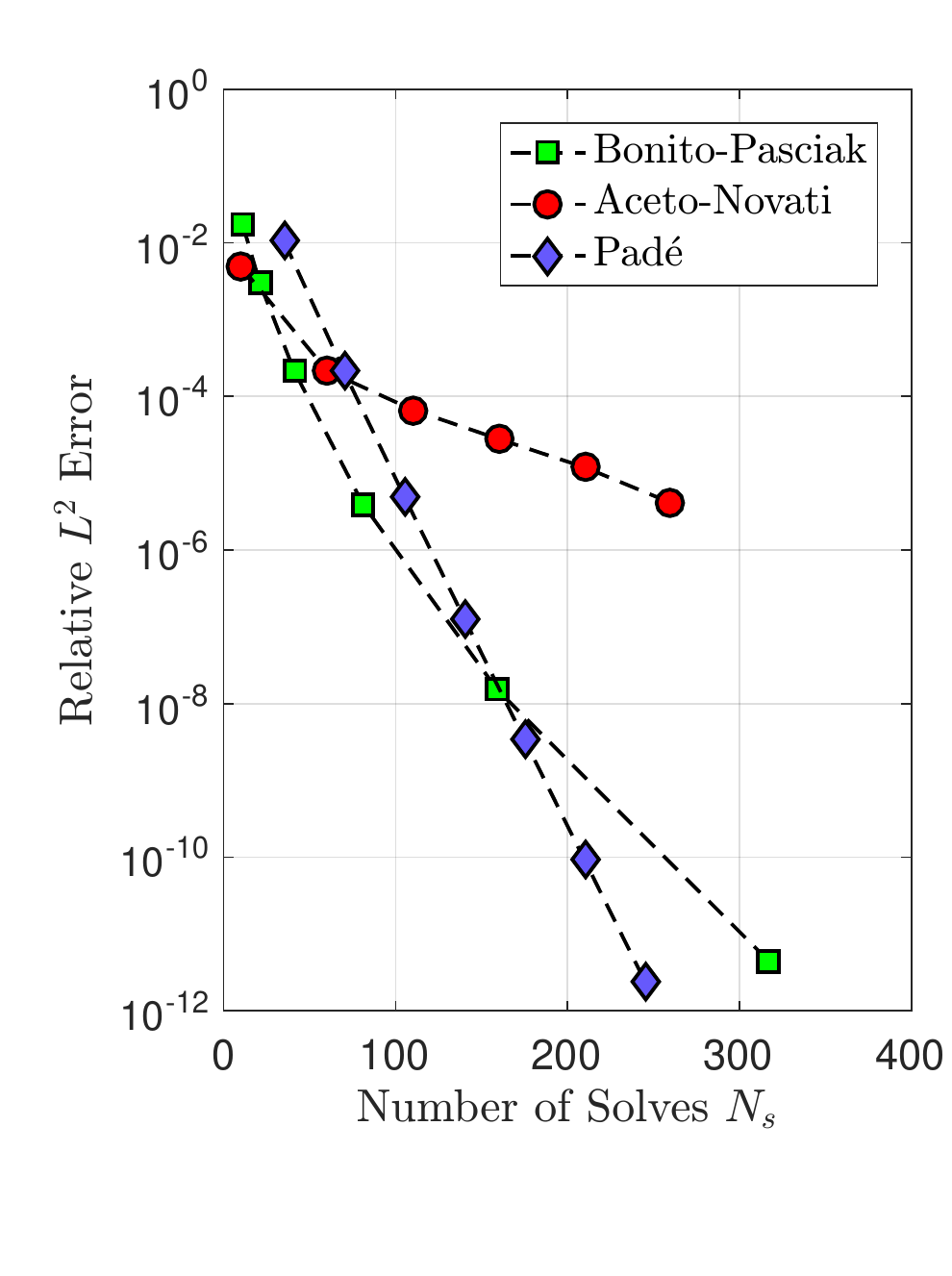}}&
		
		\subfigure[$\alpha=0.9$]{\includegraphics[width=0.32\textwidth,trim=140 210 170 250, clip]{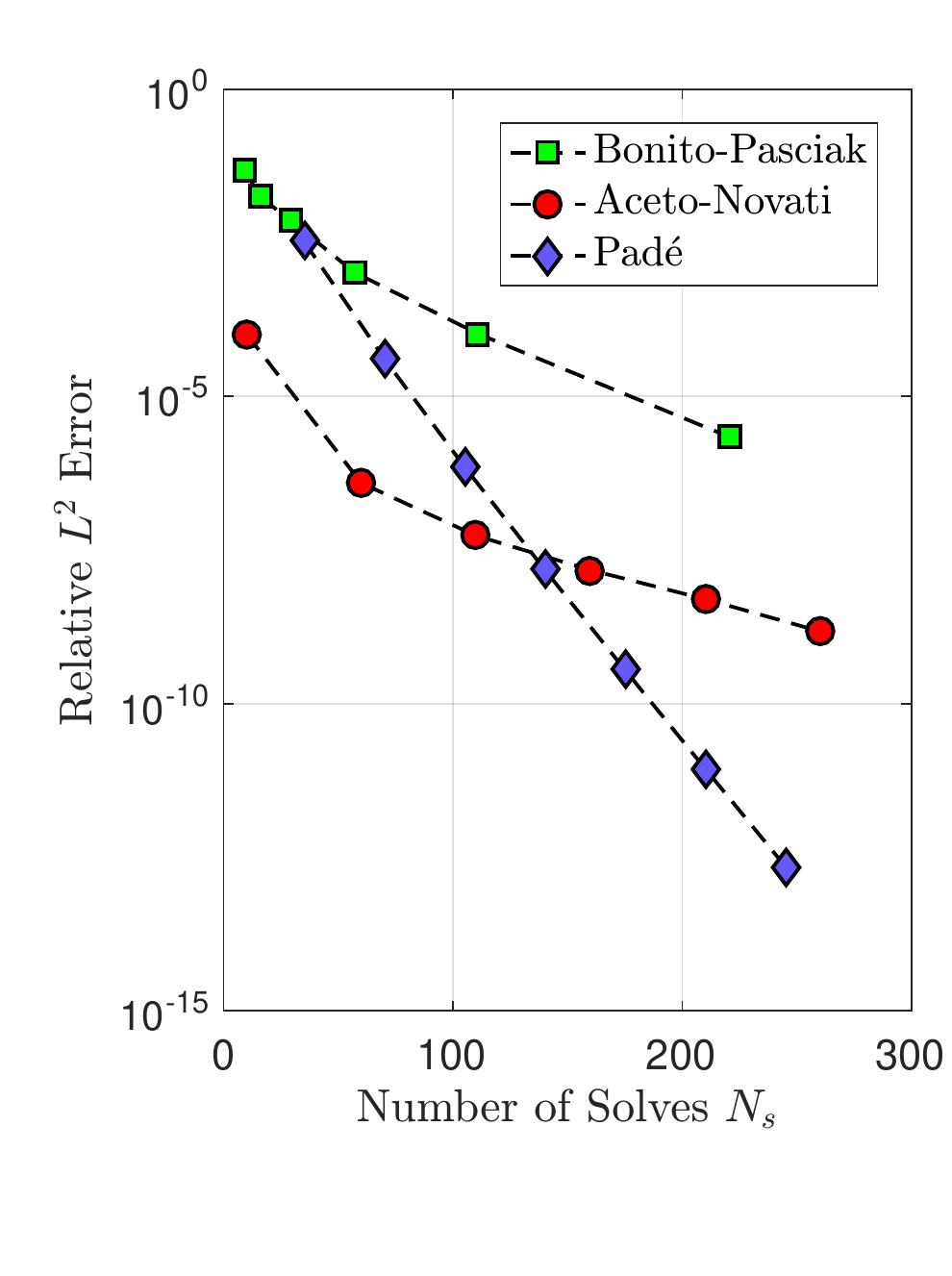}}
	\end{tabular}
	\caption{Error plots for different schemes under different $\alpha$}
	\label{fig-compare}
\end{figure}

\begin{example}\label{2d-sphere}
	 To verify Theorem \ref{The-est-h}, we consider $$\cM=\left\{(x_1,x_2,x_3)|x_1^2+x_2^2+x_3^2=1\right\},$$
	  say, $\cM$ is a  spherical surface in $\mathbb{R}^3$ with radius $r=1$.  We take $\cLg=-\Delta_\cM$  and 
	\begin{equation*}
		f=\left\{ \begin{matrix}
			1,\ \ & \mbox{for } &  x_3\ge 0;\\
%			0,\ \ & \mbox{for } &  x_3=0;\\
			-1\ \ & \mbox{for } & x_3<0,\\
		\end{matrix}
		\right.
	\end{equation*}
then $f\in \mH^{1/2-\epsilon}(\cM)$. In this case, it is known that the eigenfunctions of $\cLg$ are spherical harmonic functions and corresponding eigenvalues are $\{\lambda_n=n(n+1)\}_{n=1}^\infty$. We exclude $\lambda_0=0$ since the corresponding non-trivial constant eigenfunction does not satisfy the integral condition $\int_{\cM}\psi_0 \,d\cM=0$. Utilizing the fact that  $f$ is  constant along the latitudinal direction, and after manipulations one can write down the explicit formula for $u$:
\begin{equation*}
	u(x_1,x_2,x_3)=\sum_{n=1}^\infty \lambda_n^{-\alpha}\frac{2n+1}{2n}J_{n-1}^{1,1}(0)J^{0,0}_n(x_3).
\end{equation*}
\end{example}

We choose $f_h=\pi_h f^{-\nell}$. Table \ref{Tab-Sphere} presents $\|u^{-\ell}-u_h\|_{L^2(\cM_h)}$ under different $\cM_h$ where we use $Dof$ to represent the number of vertexes of $\cM_h$. The numerical convergence rates against $h$ are obtained by 
\begin{equation}\label{conv-rate}
conv.\, rate=	\frac{\ln \|u^{-\ell}-u_h\|_{L^2(\cM_h)} - \ln \|u^{-\ell}-u_{h'}\|_{L^2(\cM_{h'})}}{\ln \sqrt{Dof'} -\ln \sqrt{Dof}}
\end{equation}
where $u_h$ is obtained by our scheme with $m=3$ and $\tlambda=1$. The numbers in brackets are theoretical rates given in Theorem \ref{The-est-h}. One can see that the results consist well with our theoretical analysis. The three images in Figure \ref{Fig-Sphere} show the numerical solutions for $\alpha=0.01,0.5$ and $0.99$. 

\begin{table}[htbp!]
	\centering
	\caption{$\|u^{-\ell}-u_h\|_{L^2(\cM_h)}$ under different mesh}
	\label{Tab-Sphere}
	\begin{tabular}{c|cccc}
		\toprule 
		$\alpha\backslash Dof$ &    153       & 606         & 2418         & 9666       \\
		\hline
		$\alpha=0.01$          &3.8192e-01   &2.6089e-01   &1.8143e-01   &1.2642e-01  \\
		& (0.5)        &  0.55       &   0.53       & 0.52       \\
		\hline
		$\alpha=0.3$           &8.1961e-02   &3.9167e-02   &1.8433e-02   &8.6204e-03  \\
		& (1.1)        & 1.07        & 1.09         & 1.10       \\
		\hline
		$\alpha=0.5$           &2.5687e-02   &1.0153e-02   &3.7672e-03   &1.3619e-03  \\
		& (1.5)        & 1.35        & 1.43         & 1.47       \\	
		\hline	
		$\alpha=0.7$           &1.1322e-02   &3.3840e-03   &9.5737e-04   &2.6565e-04  \\
		& (1.9)        & 1.76        & 1.83         & 1.85       \\
		\hline
		$\alpha=0.99$          &1.8734e-02   &4.8750e-03   &1.2329e-03   &3.0923e-04  \\
		& (2.0)        & 1.96        & 1.99         & 2.00       \\		
		\bottomrule
	\end{tabular}
\end{table}
\begin{figure}[!ht]
	\centering
	\begin{tabular}{ccc}
	\subfigure[$\alpha=0.01$]
	{\includegraphics[width=0.32\textwidth,trim=100 230 60 220, clip]{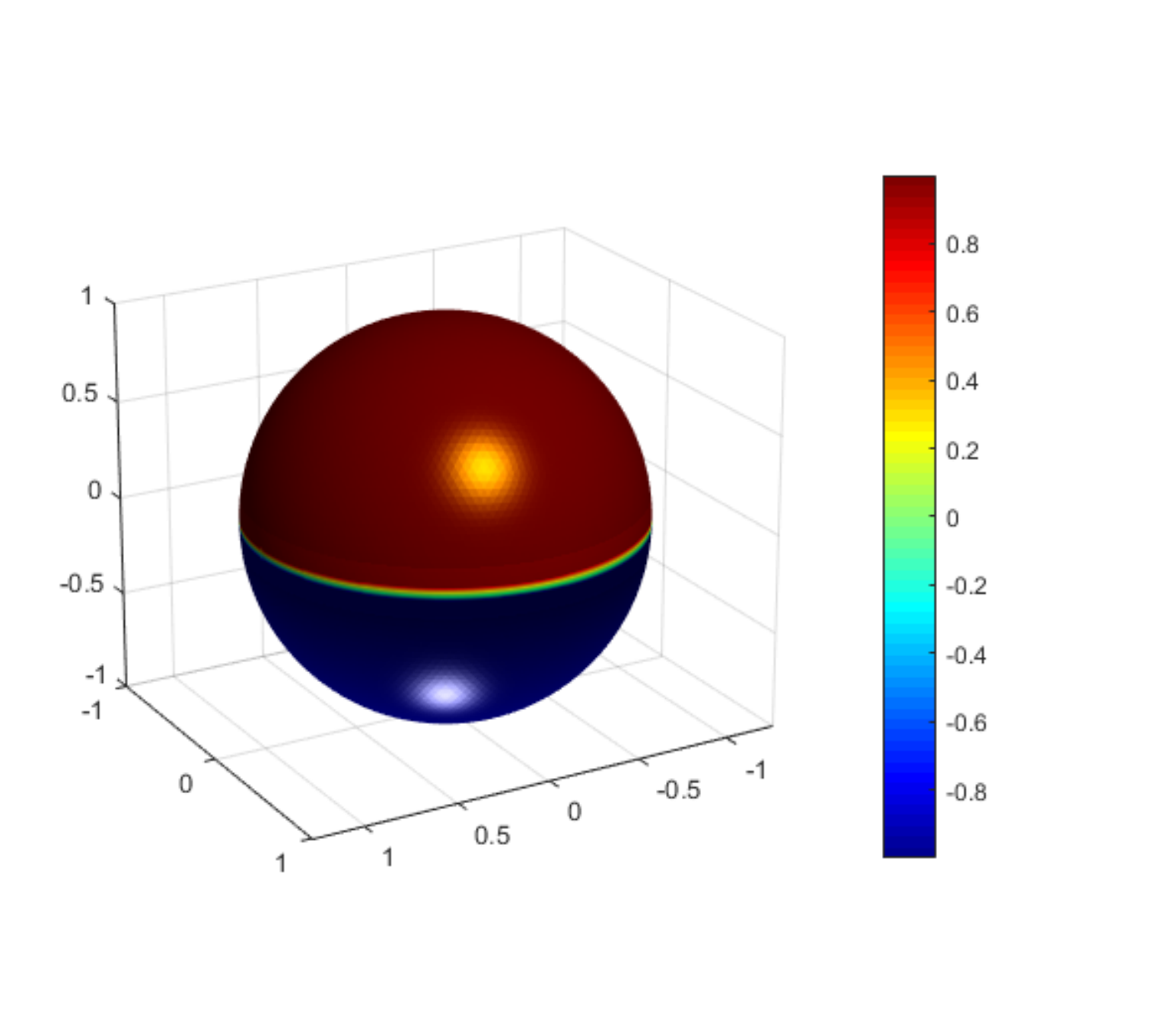}}
	\subfigure[$\alpha=0.5$]
	{\includegraphics[width=0.32\textwidth,trim=100 230 60 220, clip]{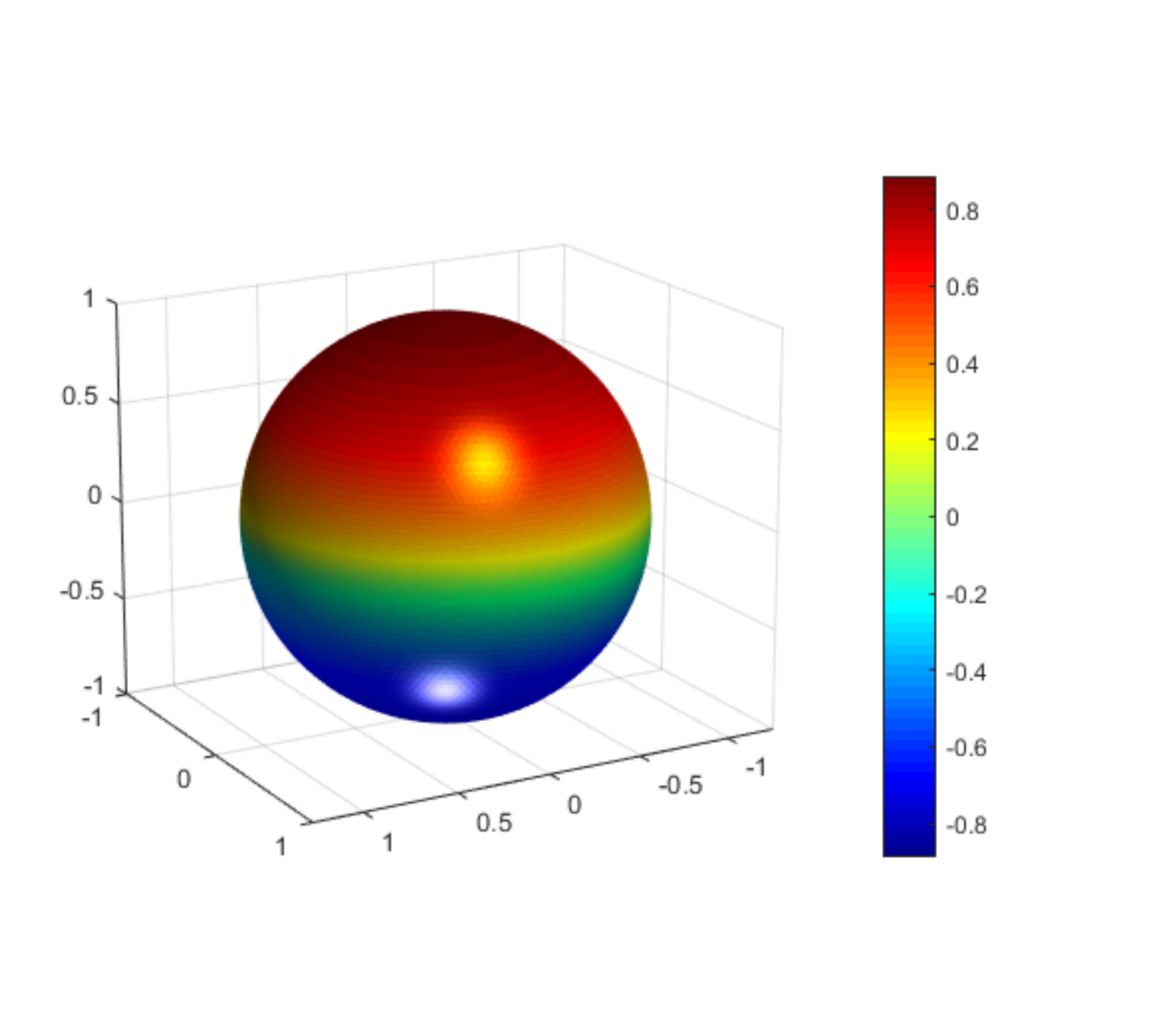}}
	\subfigure[$\alpha=0.99$]
	{\includegraphics[width=0.32\textwidth,trim=100 230 60 220, clip]{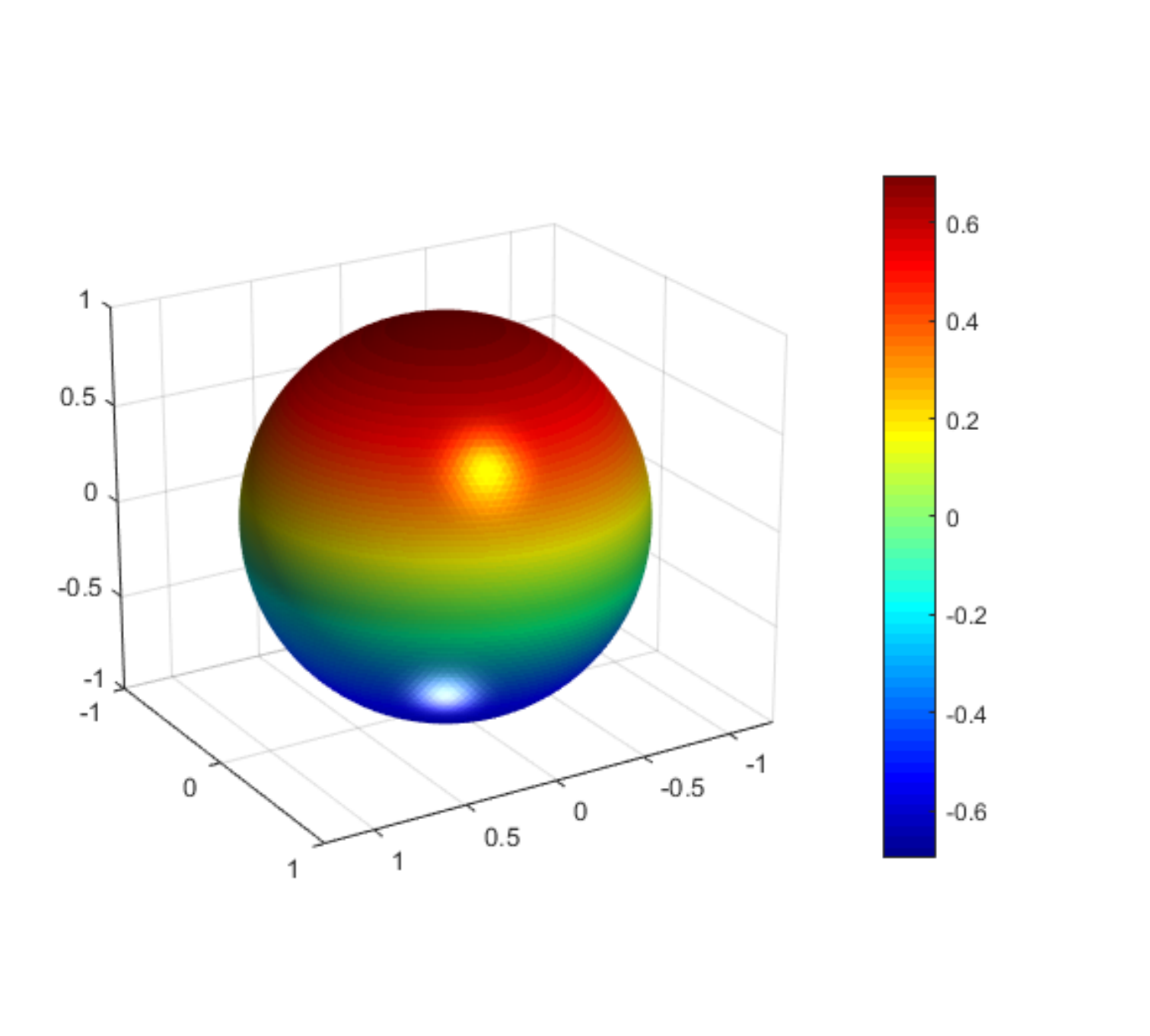}}
	\end{tabular}
	\caption {Numerical solutions for $\alpha=0.01,0.5$ and $0.99$, respectively}
	\label{Fig-Sphere}
\end{figure}

\begin{example}\label{torus}
	In this example we take a torus as $\cM$, which is given parametrically by
	\begin{equation*}
		\vec{x}=[(R+r\cos\varphi_1)\cos\varphi_2,(R+r\cos\varphi_1)\sin\varphi_2,r\sin\varphi_1],\quad \varphi_1,\varphi_2\in[0,2\pi)
	\end{equation*}
with $R=0.5,r=0.2$. We set $\cLg=-\Delta_\cM+\cI$ and 
\begin{equation*}
	f=H\cos\left(\arctan\left(\frac{x_2}{x_1}\right)\right)
\end{equation*}
with $H$ the mean curvature of the torus, see Fig.\ref{Fig-Torus-1}(a) for it. The surface is triangulated by $237568$ simplices with $118784$ vertexes. In this case we know $\lambda_{min}=1$ and we can obtain $\Lambda=1.78\times 10^{6}$ by a few iterations of power method which leads to $L+1=21$. In this case we choose $f_h=I_hf^{-\ell}$. The relative $L^2$-errors with respect to $m$ under various $\alpha$ are presented in Fig.\ref{Fig-Torus-1}(b). Solutions under different $\alpha$ are presented in Fig.\ref{Fig-Torus-2}. One can see that the error decays as we predicted in  Theorem \ref{The-est-h}.

\begin{figure}[!ht]
	\centering
	\begin{tabular}{cc}
		\subfigure[Source term $f$]
		{\includegraphics[width=0.45\textwidth,trim=100 180 60 220, clip]{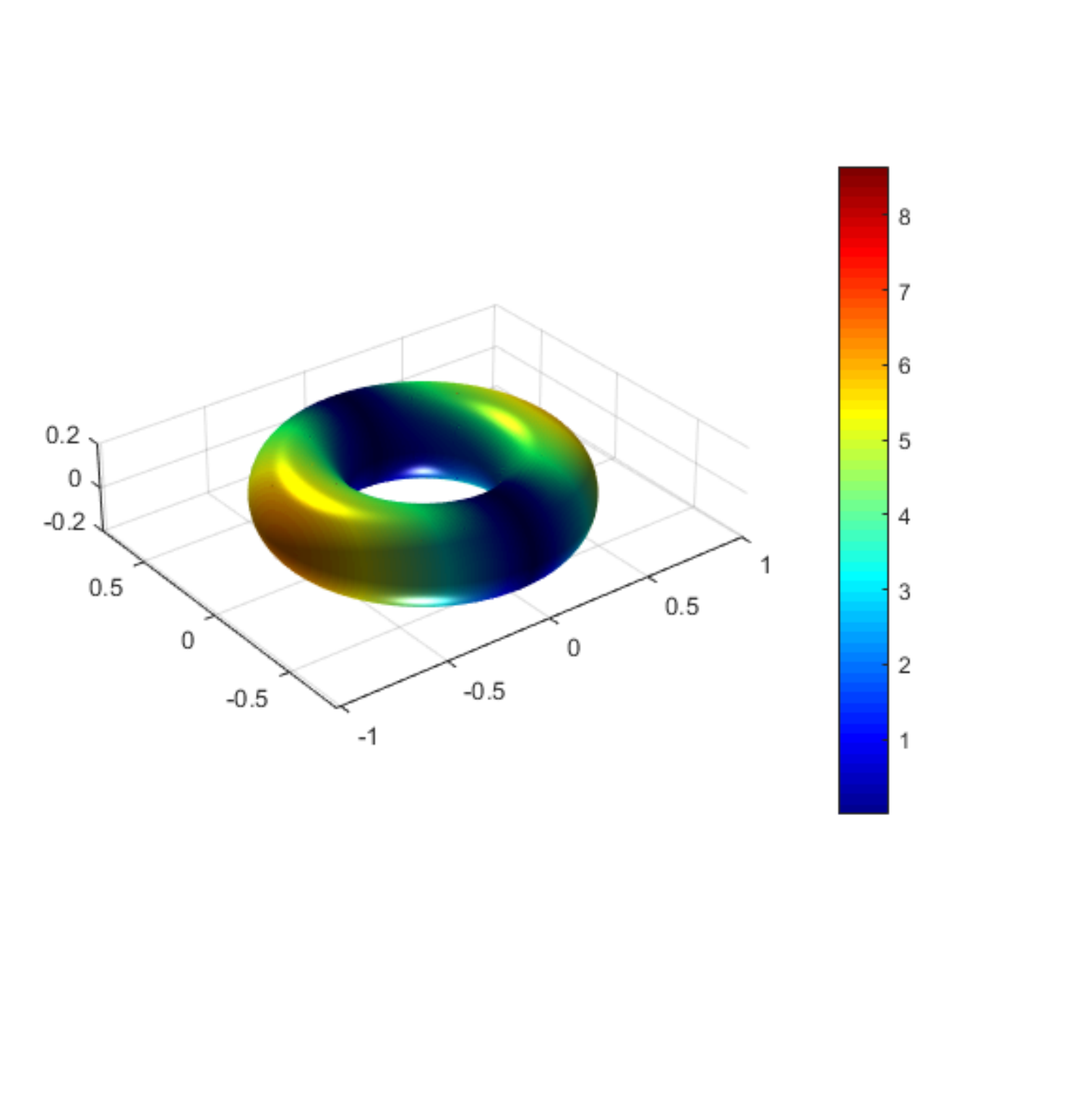}}
		\subfigure[Relative $L^2$-errors against $m$]
		{\includegraphics[width=0.45\textwidth,trim=140 210 170 230, clip]{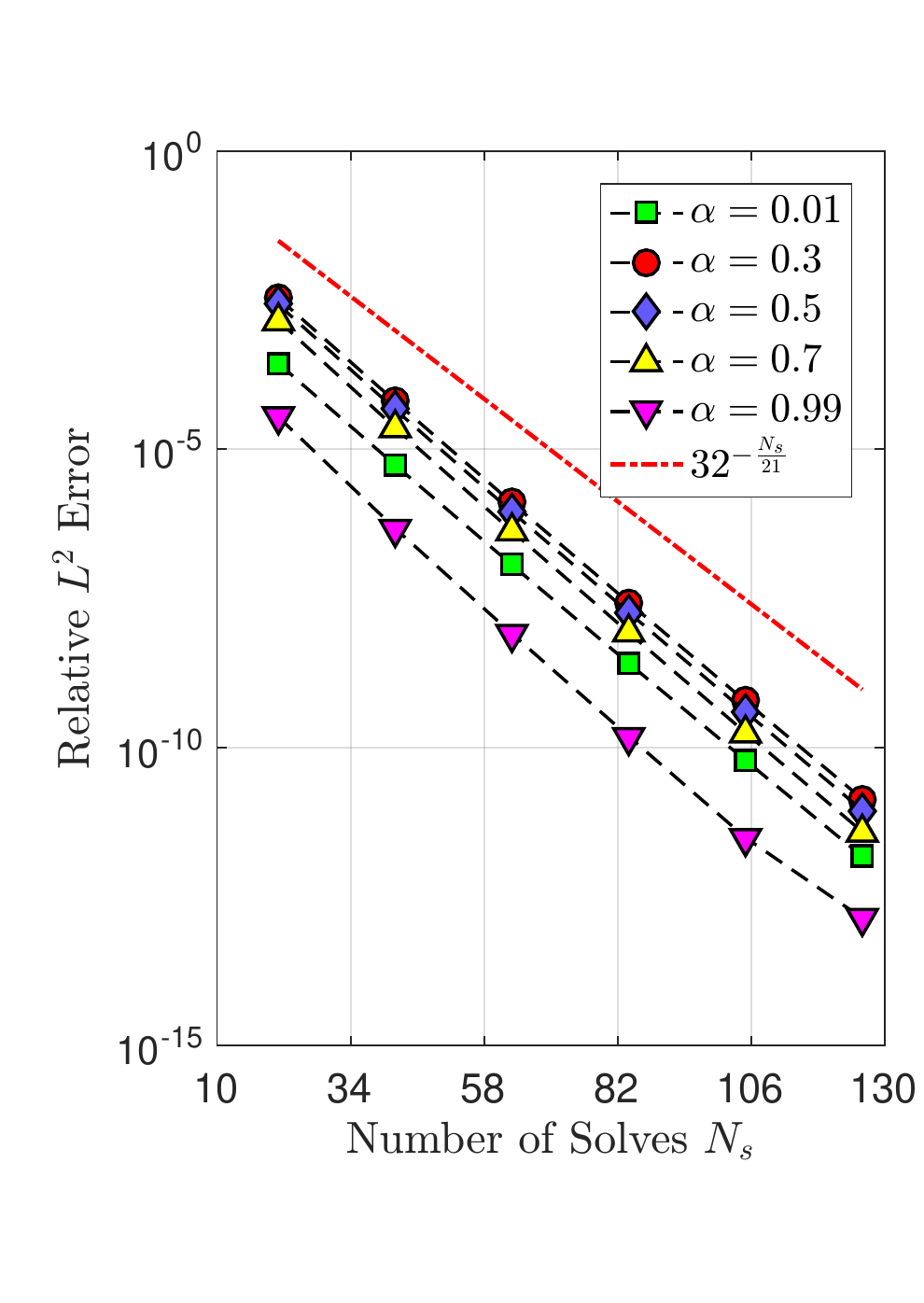}}
	\end{tabular}
	\caption {Source term and relative $L^2$-errors}
	\label{Fig-Torus-1}
\end{figure}
\begin{figure}[!ht]
	\centering
	\begin{tabular}{ccc}
		\subfigure[$\alpha=0.01$]
		{\includegraphics[width=0.32\textwidth,trim=100 230 60 220, clip]{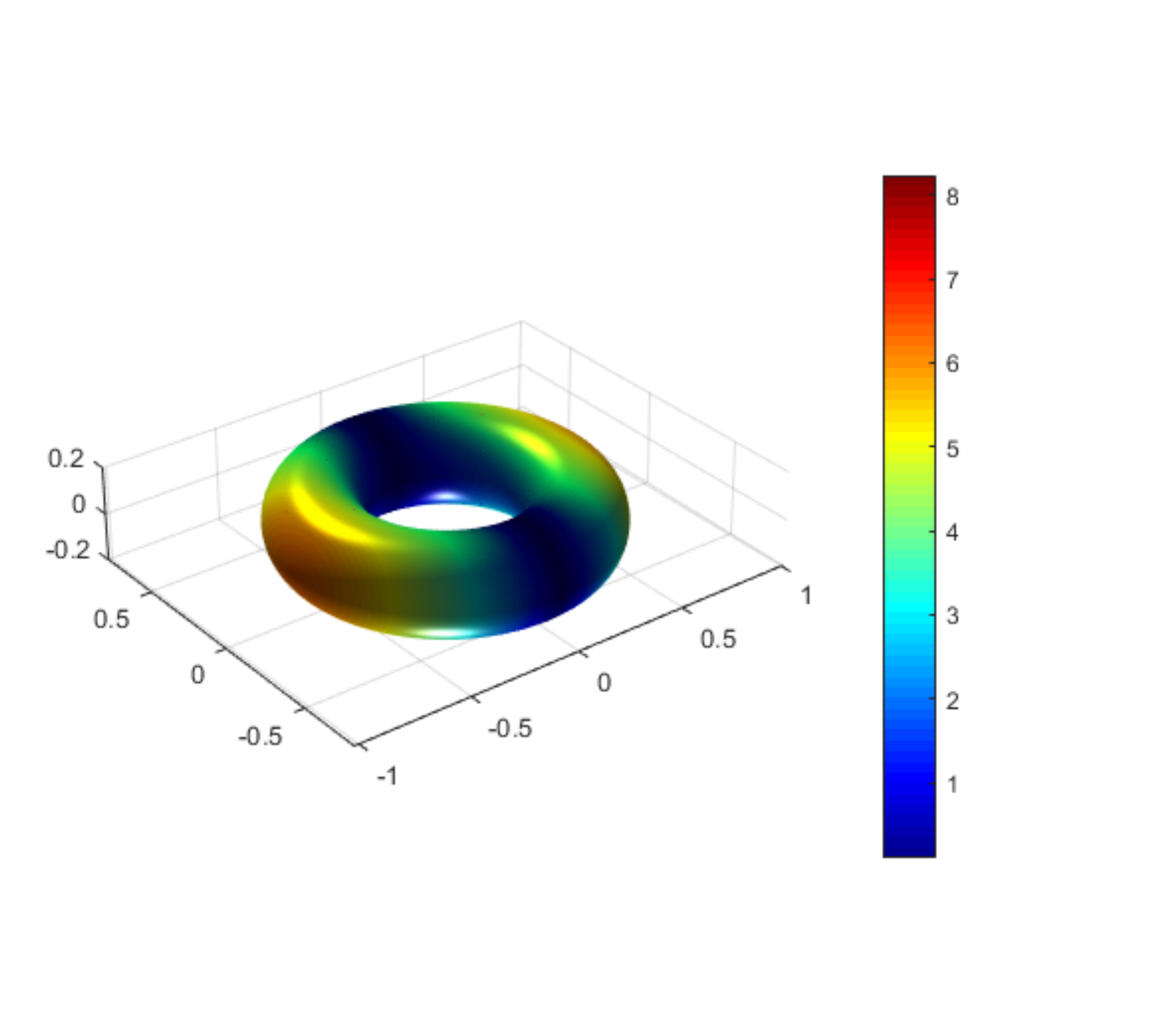}}
		\subfigure[$\alpha=0.5$]
		{\includegraphics[width=0.32\textwidth,trim=100 230 60 220, clip]{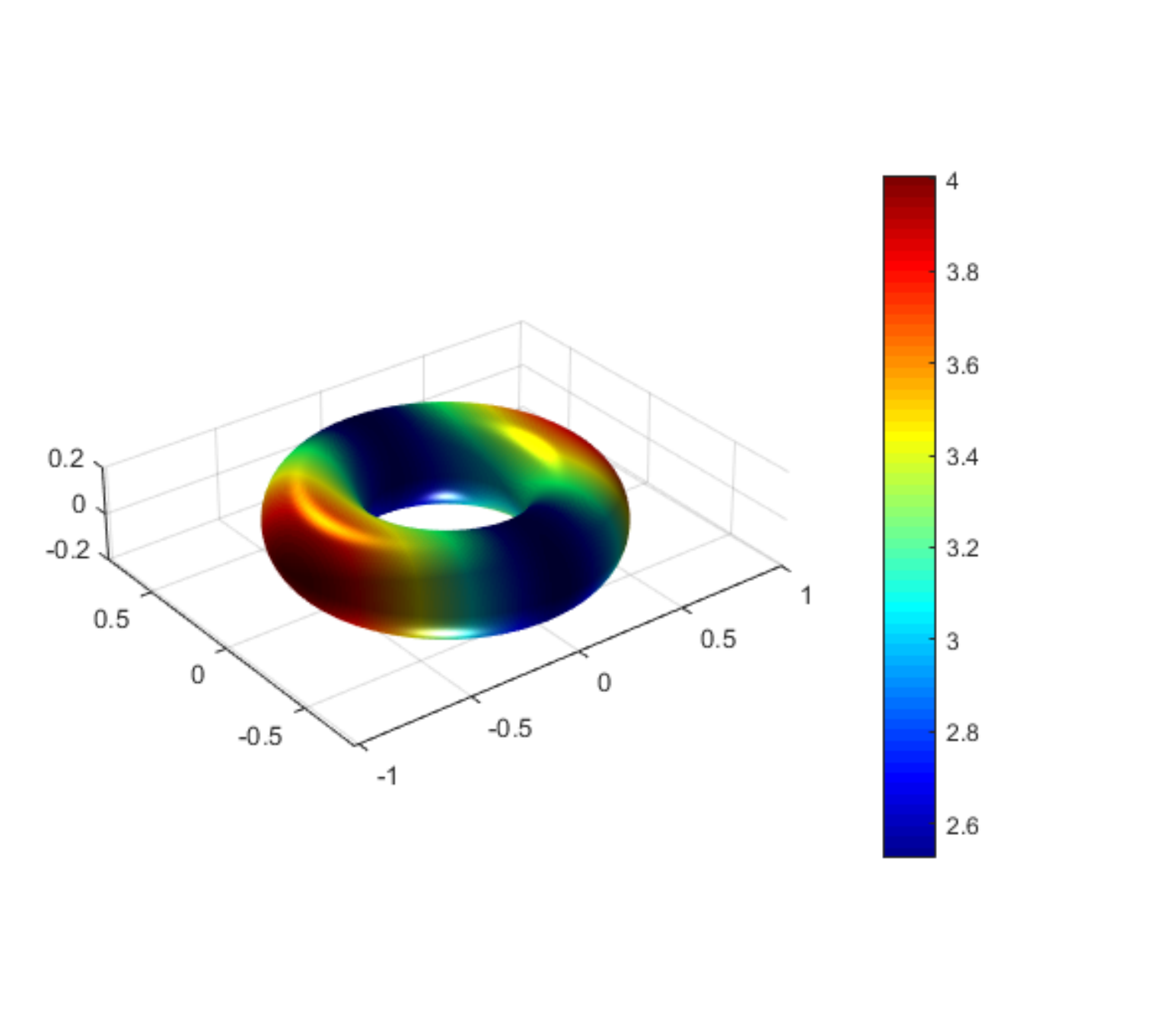}}
		\subfigure[$\alpha=0.99$]
		{\includegraphics[width=0.32\textwidth,trim=100 230 60 220, clip]{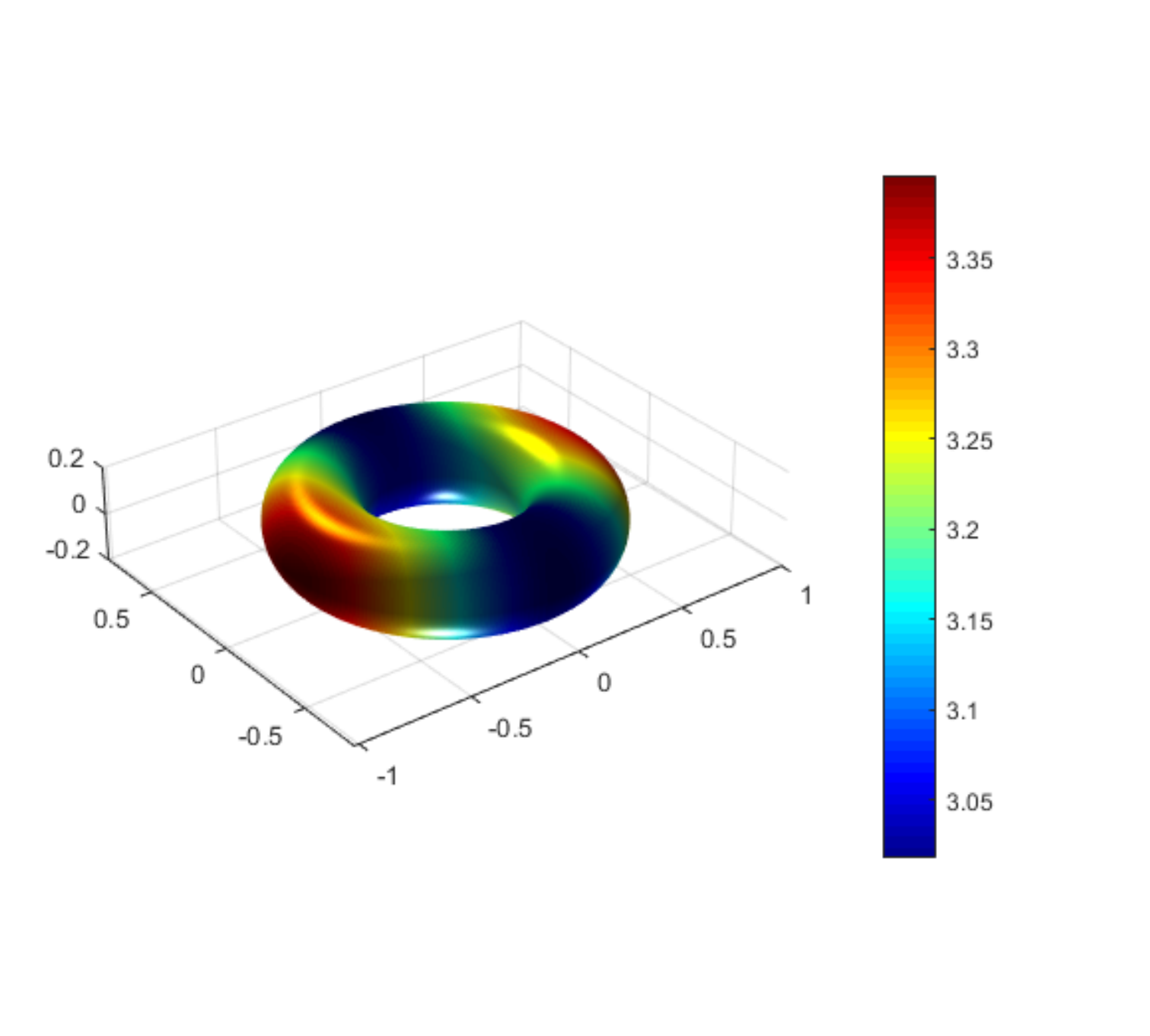}}
	\end{tabular}
	\caption {Numerical solutions for $\alpha=0.01,0.5$ and $0.99$, respectively}
	\label{Fig-Torus-2}
\end{figure}
\end{example}

\section*{Conclusion}
In this paper we proposed a Pad\'e-parametric FEM scheme to approximate fractional powers of self-adjoint operators on manifolds. Rigorous error analysis was carried out and sharp error bound for the fully discrete scheme was presented. It is also worth to point out that our method is robust with respect to $\alpha$. Several numerical tests verified the theoretical analysis and also demonstrated the efficiency and robustness of our method.

\section*{Acknowledgement}
The author would like to express sincere thanks to Prof. Raytcho Lazarov and Prof. Joseph Pasciak who initialized this Pad\'e approximation project four years ago, and some ideas in Section \ref{Sec-2} are also from them.
% Besides, many thanks to the editor and the two anonymous referees whose suggestions and comments improve this paper a lot.

%\bibliographystyle{abbrv}
\bibliographystyle{amsplain}
\bibliography{references}
%,Ray_references
\end{document}